\documentclass[preprint,12pt]{elsarticle}

\usepackage{amssymb}
\usepackage{amsthm}
\usepackage{amsmath}
\usepackage{tikz,tikz-3dplot}
\usepackage{multirow}
\usepackage{tabu}
\usepackage{tabularx}
\usepackage{float}
\usepackage[utf8]{inputenc} 
\usepackage[T1]{fontenc}
\usetikzlibrary{positioning}

\usepackage{lineno}

\newtheorem{dummy}{Dummy}[section]
\newtheorem{proposition}[dummy]{Proposition}
\newtheorem{lemma}[dummy]{Lemma}
\newtheorem{theorem}[dummy]{Theorem}
\newtheorem*{theorem*}{Theorem}
\newtheorem{corollary}[dummy]{Corollary}
\newtheorem{definition}[dummy]{Definition}
\newtheorem{remark}[dummy]{Remark}
\newdefinition{example}[dummy]{Example}
\newdefinition{examples}[dummy]{Examples}

\newcolumntype{C}{>{\centering\arraybackslash}X}

\newcolumntype{P}[1]{>{\centering\arraybackslash}p{#1}}

\makeatletter
\tikzset{my loop/.style =  {to path={
  \pgfextra{}
  [looseness=12,min distance=10mm]
  \tikz@to@curve@path},font=\sffamily\small
  }}  
\makeatletter

\makeatletter
\def\ps@pprintTitle{%
  \let\@oddhead\@empty
  \let\@evenhead\@empty
  \def\@oddfoot{\reset@font\hfil\thepage\hfil}
  \let\@evenfoot\@oddfoot
}
\makeatother

\begin{document}

\begin{frontmatter}

\title{On the lattice of fuzzy rough sets \tnoteref{fr_title}}

\author[1]{Dávid Gégény\corref{cor_gd}\fnref{fund}}
\ead{david.gegeny@uni-miskolc.hu}
\ead[url]{http://www.uni-miskolc.hu/~matgd/}

\author[1]{Sándor Radeleczki\fnref{fund}}
\ead{sandor.radeleczki@uni-miskolc.hu}
\ead[url]{http://www.uni-miskolc.hu/~matradi/}

\address[1]{Institute of Mathematics, University of Miskolc, 3515 Miskolc-Egyetemváros, Hungary}

\cortext[cor_gd]{Corresponding author}

\fntext[fund]{\textsc{Supported by the ÚNKP-22-4 New National Excellence Program of the Ministry for Innovation and Technology from the source of the National Research, Development and Innovation Fund.}}

\begin{abstract} By the means of lower and upper fuzzy
approximations we define quasiorders. Their properties are used to prove our main results. First, 
we characterize those pairs of fuzzy sets which form fuzzy rough sets w.r.t. a t-similarity relation $\theta$ on $U$, for certain t-norms and implicators. Then we establish conditions under which fuzzy rough sets form lattices. We show that for the $\min$ t-norm and any
S-implicator defined by the $\max$ co-norm with an involutive negator, the fuzzy rough
sets form a complete lattice, whenever $U$ is finite or the range of
$\theta$ and of the fuzzy sets is a fixed finite chain.  
\end{abstract}






\begin{keyword}
Fuzzy rough sets; Fuzzy relations; Lower and upper approximation; Self-dual poset 
\end{keyword}

\end{frontmatter}

\section{Introduction}
\label{sec:intro}

Rough sets were introduced by Zdzis{\l}aw Pawlak \cite{Pawlak1}, by defining the lower and upper approximations of a (crisp) set based on a so-called indiscernibility relation of the elements. 
 Originally, Pawlak assumed that this relation is an
equivalence, but later several other types of relations were also examined
(see e.g. \cite{JR1}, \cite{JRV09}, or \cite{JPR}, \cite{YY}).
For a relation $\varrho \subseteq U\times U$ and any element $u\in
U$, denote $\varrho(u):=\{x\in
U\mid(u,x)\in R\}$. Now, for any subset $A\subseteq U$, the \emph{lower
approximation} of $A$ is defined as
\[
A_{\varrho}:=\{x\in U\mid \varrho(x)\subseteq A\},
\]
and the \emph{upper approximation} of $A$ is given by
\[
A^{\varrho}:=\{x\in U\mid \varrho(x)\cap A\neq\emptyset\}.
\]

\noindent If $\varrho$ is reflexive and transitive, i.e. it is a
\emph{quasiorder}, then the properties $A_{\varrho}\subseteq A\subseteq A^{\varrho}$ and
$(A_{\varrho})_{\varrho}=A_{\varrho}$, $\left(  A^{\varrho}\right)  ^{\varrho}=A^{\varrho}$ hold for all
$A\subseteq U$.
\medskip
\noindent  The rough sets induced by $ \varrho $ can be ordered w.r.t. the component-wise inclusion, and for an equivalence, or more generally, for a quasiorder $ \varrho $, they form a complete distributive lattice  with several particular properties, see e.g. \cite{JRV09} or \cite{Po88}.

The notion of a fuzzy set was introduced by Lotfi Zadeh \cite{Zh}.  A fuzzy set is defined by a mapping $ f: U \rightarrow [0, 1]$. We say that $f$ has a \emph{finite range}, whenever
the (crisp) set $\{f(x)\mid x\in U\}$ is finite. The collection of all fuzzy
sets on $U$ is denoted by $\mathcal{F}(U)$. Ordering any elements
$f,g\in\mathcal{F}(U)$ as follows%
\[
f\leq g\Leftrightarrow f(x)\leq g(x)\text{, for all }x\in U\text{,}%
\]

\noindent we obtain a completely distributive (complete) lattice
$\mathcal{F}(U)$. For any system $f_{i}\in\mathcal{F}(U)$, $i\in I$, its
infimum and supremum are given by the formulas%
\begin{equation}
\left(  \bigwedge_{i\in I}f_{i}\right)  (x)=\bigwedge_{i\in I}f_{i}(x)\text{;
}\left(  \bigvee_{i\in I}f_{i}\right)  (x)=\bigvee_{i\in I}f_{i}(x)\text{,}
\tag{1}%
\end{equation}

\noindent where $\bigwedge$ and $\bigvee$ denote the infimum and the supremum, respectively, in the complete lattice $([0,1],\leq)$.

The first step to integrate the two main theories relates to the works of 
del Cerro and Prade \cite{CP}, Nakamura \cite{NK} and Dubois and Prade \cite{DP}. In \cite{DP} the fuzzy rough sets are defined as pairs $(\underline{f},\overline{f}) \in\mathcal{F}(U)\times
\mathcal{F}(U)$ of lower and upper approximations of the fuzzy sets $f \in \mathcal{F}(U)$.  These fuzzy approximations were defined by using a similarity relation, the t-norm min and conorm max.  Their approach was generalized in several papers, like \cite{CCK, GMS, HYPC, In, Pei, RK, WLM, WMZ} and \cite{DVCG}, where fuzzy rough sets are defined on the basis of different t-norms (or conjunctors) and related implicators.
A detailed study of these approximation operators was developed in \cite{DVCG}, \cite{RK} and in \cite{MH}, \cite{RKL}, where the structure of the lower and upper approximations of $L$-fuzzy sets is also investigated. An axiomatic approach of these properties was elaborated e.g. in \cite{BJR}, \cite{Liu}, \cite{MZ},  and \cite{MY}. In \cite{GKR} it was shown that for crisp reference sets (i.e. for
$f(x)\in\{0,1\}$, $\forall x\in U$) the fuzzy rough sets defined by a
t-similarity relation $\theta$ with a well-ordered spectrum form a completely distributive lattice. 

The goal of the present paper is to find conditions under which fuzzy rough sets form
lattices. With this purpose, by the means of lower and upper fuzzy
approximations we define (crisp) quasiorders on $U$. The properties of these
quasiorders and of the equivalences determined by them are discussed
in Sections \ref{sec:quasiorders}, \ref{sec:eqbyqo} and \ref{sec:propofeqc}. These properties will be used to prove our main results, Theorems \ref{thm:frs} and \ref{thm:lattice}. Section \ref{sec:prelim} contains the essential prerequisites of our study. In Section \ref{sec:frs} by using singleton equivalence classes,
we characterize those pairs of fuzzy sets which form a fuzzy rough set with
respect to a t-similarity relation $\theta$ for certain t-norms and related implicators. In Section \ref{sec:latticefrs}, we
establish conditions under which fuzzy rough sets with a finite range form
lattices. For instance, we show that for the min t-norm and any
S-implicator defined by the max co-norm with an involutive negator, the fuzzy rough
sets form a complete lattice, whenever $U$ is finite, or whenever the range of
$\theta$ and of the fuzzy reference sets is a fixed finite chain $L\subseteq
\lbrack0,1]$.
\section{Preliminaries}
\label{sec:prelim}

\subsection{T-norms, implicators and fuzzy relations}
\medskip

\emph{A triangular norm} $\odot$ (\emph{t-norm} for short) is a commutative,
associative and monotone increasing binary operation $\odot$
defined on $[0,1]$ satisfying $1\odot x=x\odot1=x$, for all $x\in\lbrack0,1]$.
The t-norm $\odot$ is called \emph{(left) continuous,} if it is (left)
continuous as a function $\odot\colon\lbrack0,1]^{2}\rightarrow\lbrack0,1]$
in the usual interval topology on $[0,1]^{2}$. Every t-norm $\odot$ satisfies
$x\odot0=0\odot x=0$, for all $x\in L$. The most known t-norms are:

\noindent - the \emph{standard min operator}: $x\odot y:= \min(x,y)$; \\
- the \emph{arithmetical product}: $x\odot y:= x\cdot y$; \\
- the \emph{\L ukasiewicz t-norm} $x\odot y:= \max(0,x+y-1)$. \\
A \emph{negator} is a decreasing map $n\colon\lbrack0,1]\rightarrow
\lbrack0,1]$ with $n(0)=1$ and $n(1)=0$. $n$ is called
\emph{involutive} if $n(n(x))=x$, for all $x\in\lbrack0,1]$ (see e.g. \cite{Fodor}). The so-called
\emph{standard negator} $n(x):=1-x$, $x\in\lbrack0,1]$ is an involutive negator. 

\emph{A triangular conorm} $\oplus$ (shortly \emph{t-conorm}) is a
commutative, associative and monotone increasing binary operation $\oplus$
defined on $[0,1]$, that satisfies $0\oplus x=x\oplus0=x$, for all
$x\in\lbrack0,1]$. The t-conorm $\oplus$ is \emph{(left) continuous,} if it is
(left) continuous as a function $\oplus\colon\lbrack0,1]^{2}\rightarrow
\lbrack0,1]$ in the usual topology.

Given an involutive negator $n$, a t-norm $\odot$ and a t-conorm $\oplus$, we
say that $\odot$ and $\oplus$ form an $n$\emph{-dual pair} if for all $x,y \in [0,1]$%
\[
n(x\oplus y)=n(x)\odot n(y)\text{.}%
\]

\noindent Clearly, this identity also implies the identity%
\[
n(x\odot y)=n(x)\oplus n(y)\text{.}%
\]

\noindent For instance, min$(x,y)$, max$(x,y)$ form a well-known $n$-dual pair
w.r.t. any involutive negator on $[0,1]$.

An \emph{implicator} is a binary operation (mapping) $\vartriangleright
\colon\lbrack0,1]^{2}\rightarrow\lbrack0,1]$ that is decreasing in the
first and increasing in the second argument and that satisfies the boundary conditions%

\[
0\vartriangleright0=0\vartriangleright1=1\vartriangleright1=1\text{ and
}1\vartriangleright0=0\text{.}%
\]

\noindent$\vartriangleright$ is called a \emph{border implicator} if
$1\vartriangleright x=x$ holds for all $x\in\lbrack0,1]$. There are two
important classes of border implicators. The \emph{R-implicato}r based on a
t-norm $\odot$ is defined by
\[
x\vartriangleright y:=\bigvee\{z\in\lbrack0,1]\mid x\odot z\leq y\}\text{, for
all }x,y\in\lbrack0,1].
\]

\noindent If $\odot$ is a
continuous t-norm, then the algebra $([0,1],\vee,\wedge,\odot,\vartriangleright
,0,1)$ is a so-called \emph{commutative (integral) residuated lattice} (see \cite{GJ}).
Let $\oplus$ be a t-conorm and $n$ a negator on $[0,1]$, then the
\emph{S-implicato}r based on them is defined by
\[
x\vartriangleright y:=n(x)\oplus y
\]

\noindent The\emph{ \L ukasiewicz implicator} $\vartriangleright
_{L}$ is both an R-implicator and S-implicator defined by $x$
$\vartriangleright_{L}y:=$ min$(1,1-x+y)$, $ \forall x\in\lbrack0,1]$. The
\emph{Kleene-Dienes (KD) implicator} $\vartriangleright_{KD}$\ is an
S-implicator given by $x\vartriangleright_{KD}y$ $:=\ $max$(1-x,y)$,
$ \forall x\in\lbrack0,1]$. If $\vartriangleright$ is an implicator, then a
corresponding negator is defined by $n(x)=x\vartriangleright0$.
If $n$ is an involutive negator and $\vartriangleright$ is an
R-implicator defined by a left-continuous t-norm, then
$\vartriangleright$ is called an \emph{ITML-implicator.}

\medskip

\noindent A \emph{fuzzy binary relation} on $U$ is a fuzzy set $\theta\colon
U\times U\rightarrow\lbrack0,1]$. The pair $(U,\theta)$ is usually called a
\emph{fuzzy approximation space}. $\theta$\ is called \emph{reflexive} if
$\theta(x,x)=1$ for all $x\in U$, and it is called \emph{symmetric} if for all
$x,y\in U$, $\theta(x,y)=\theta(y,x)$. Given a t-norm $\odot$, the relation
$\theta$ is called $\odot$\emph{-transitive} if
\[
\theta(x,y)\odot\theta(y,z)\leq\theta(x,z)
\]

\noindent holds for every $x,y,z\in U$. If a relation $\theta$ is reflexive
and $\odot$-transitive, then it is called a (fuzzy)\emph{ }$\odot
$\emph{-quasiorder.} A symmetric $\odot$-quasiorder\emph{ }$\theta$ is called
a (fuzzy) $\odot$\emph{-similarity relation}. When $x\odot y=$ min$(x,y)$,
then $\theta$ is simply named a \emph{similarity relation}. Since the minimum
t-norm is the largest t-norm, a similarity relation is always $\odot
$-transitive for any t-norm $\odot$. We say that $\theta$ is of a \emph{finite
range}, if the (crisp) set $\{\theta(x,y)\mid x,y\in U\}$ of its values is finite.

\subsection{Fuzzy rough sets}

\noindent Let $(U,\theta)$ be a fuzzy approximation space with a relation $\theta\colon U\times U\rightarrow\lbrack0,1]$. The precise notion of a fuzzy
rough set was introduced by D. Dubois and H. Prade in  \cite{DP}. They defined for any fuzzy set $f\in\mathcal{F}(U)$ its \emph{lower}
\emph{approximation} $\underline{\theta}(f)$ and its \emph{upper
approximations }$\overline{\theta}(f)$ \emph{relative to }$\theta$ by the formulas%

\begin{center}

$\underline{\theta}(f)(x):=\bigwedge\{\text{max}(1-\theta(x,y),f(y))\mid y\in
U\}\text{, for all }x\in U\text{;} $%

$\overline{\theta}(f)(x):=\bigvee\{\text{min}(\theta(x,y),f(y))\mid y\in
U\}\text{, for all }x\in U\text{.}$

\end{center}

\noindent The \emph{fuzzy rough set of} $f$ is identified by the pair
$(\underline{\theta}(f),\overline{\theta}(f))\in\mathcal{F}(U)\times
\mathcal{F}(U)$ (see \cite{DP}). This definition was generalized in several papers.
Here we will use the approach based on implicators and t-norms from \cite{DVCG} and \cite{RK}. Hence, in what follows, let $\odot$ be a t-norm and
$\vartriangleright$ a border implicator on $[0,1]$.

\begin{definition}
If $(U,\theta)$ is a fuzzy approximation
space, then for any fuzzy set $f\in\mathcal{F}(U)$ its \emph{fuzzy lower
approximation} $\underline{\theta}(f)$ and its \emph{fuzzy upper
approximation} $\overline{\theta}(f)$ are defined as follows:%

\begin{equation}
\underline{\theta}(f)(x):=\bigwedge\{\theta(x,y)\vartriangleright f(y)\mid
y\in U\}\text{, for all }x\in U\text{.}\tag{2}%
\end{equation}

\begin{equation}
\overline{\theta}(f)(x):=\bigvee\{\theta(x,y)\odot f(y)\mid y\in U\}\text{,
for all }x\in U\text{.}\tag{2'}%
\end{equation}

\noindent The pair $(\underline{\theta}(f),\overline{\theta}(f))\in
\mathcal{F}(U)\times\mathcal{F}(U)$ is called a \emph{fuzzy rough set in}
$(U,\theta)$.
\end{definition}

\noindent This definition also includes the one of Dubois
and Prade, where $\odot$ is the min t-norm and $\vartriangleright$ is
the Kleene-Dienes implicator $x\vartriangleright_{KD}y$ $=\ $max$(1-x,y)$.

Notice that $\underline{\theta}$ and $\overline{\theta}$ are
\emph{order-preserving} operators, i.e. $ f\leq
g$ implies $\underline{\theta}(f)\leq\underline{\theta}(g)$ and $\overline
{\theta}(f)\leq\overline{\theta}(g)$. In addition, if $\theta$ is a reflexive
fuzzy relation, then $\underline{\theta}(f)\leq f\leq\overline{\theta}(f)$
holds for all $f\in\mathcal{F}(U)$ (see e.g. \cite{DVCG} or \cite{RK}) The following
properties will have a special importance in our proofs:

\medskip

(D) Let $\odot$ be a left-continuous t-norm such that its induced R-implicator
$\vartriangleright$ is an ITML implicator, i.e. $n(x):=x\vartriangleright0$,
$x\in U$ is an involutive negator, or let $n$ be an involutive negator, $\oplus$
a t-conorm $n$-dual to $\odot$ and $\vartriangleright$ the S-implicator
defined by them (i.e. $x\vartriangleright y=n(x)\oplus y$). Then
$n(\overline{\theta}(f))=\underline{\theta}(n(f))$ and $n(\underline{\theta
}(f))=\overline{\theta}(n(f))$ (see e.g. \cite{DVCG}, \cite{MH} or \cite{RK}).

(ID) Let $\odot$ be a left-continuous t-norm and $\vartriangleright$
the R-implicator induced by it, or $n$ an involutive negator, $\oplus$ a
t-conorm $n$-dual to $\odot$ and $\vartriangleright$ the S-implicator
corresponding to them. If $\theta$ is $\odot$-transitive, then for any
$f,g\in\mathcal{F}(U)$ we have $\overline{\theta}(\overline{\theta
}(f))=\overline{\theta}(f)$ and $\underline{\theta}(\underline{\theta
}(g))=\underline{\theta}(g)$ (see \cite{DVCG}, \cite{MH}, \cite{RK}). In other words, for
$F=\overline{\theta}(f)$ and $G=\underline{\theta}(g)$ we have $F=\overline
{\theta}(F)$ and $G=\underline{\theta}(G)$.

\begin{lemma}
\label{lem:frange}
Let $(U,\theta)$ be a fuzzy
approximation space such that the relation $\theta$ is of a finite
range. If $f\in\mathcal{F}(U)$ has a finite range, then the fuzzy sets
$\underline{\theta}(f)$ and $\overline{\theta}(f)$  are also of a
finite range.
\end{lemma}

\begin{proof}
Since $\{\theta(x,y)\mid x,y\in U\}$ and
$\{f(y)\mid y\in U\}$ are finite sets, their Cartesian product $\{(\theta
(x,y),f(y))\mid x,y\in U\}$ is finite, hence the sets $\mathcal{C}%
=\{\theta(x,y)\odot f(y)\mid x,y\in U\}$ and $\mathcal{I}=\{\theta
(x,y)\vartriangleright f(y)\mid x,y\in U\}$ are also finite. In particular,
this means that the sets $\mathcal{C}$ and $\mathcal{I}$ have finitely many
(different) subsets of the form $\{\theta(x,y)\vartriangleright f(y)\mid y\in
U\}$ and $\{\theta(x,y)\odot f(y)\mid y\in U\}$, and this immediately implies
that both $\underline{\theta}(f)$ and $\overline{\theta}(f)$ have finitely
many different values, i.e. they have finite ranges.
\end{proof}

\section{Quasiorders induced by lower and upper approximations}
\label{sec:quasiorders}

In what follows, suppose that conditions in (ID) hold and $n(x):=x\vartriangleright0$. For any $f,g\in\mathcal{F}(U)$, denote $F=\overline{\theta}(f)$ and $G=\underline
{\theta}(g)$.  Using $F$ and $G$ we define two binary relations $R(F)$ and
$\varrho(G)$ on $U$ as follows:

\begin{definition}
Let $(U,\theta)$ be a fuzzy approximation
space, $a,b\in U$ and $F=\overline{\theta}(f)$, $G=\underline{\theta}(g)$. Then

\noindent(i) $(a,b)\in R(F)\Leftrightarrow F(a)=$ $\theta(a,b)\odot F(b)$;

\noindent(ii) $(a,b)\in\varrho(G)\Leftrightarrow G(a)=$ $\theta
(a,b)\vartriangleright G(b)$.
\end{definition} 

\begin{proposition}
\label{prop:prop1}
(i) $(a,b)\in R(F)$ implies $F(a)\leq
\theta(a,b)$, and $(a,b)\in\varrho(G)$ implies $G(a)\geq n(\theta(a,b))$.

\noindent(ii) If $\theta$ is reflexive, then for any $f,g\in\mathcal{F}(U)$,
$R(F)$ and $\varrho(G)$ are reflexive.

\noindent(iii) If $\theta$ is a $\odot$-quasiorder, then $R(F)$, $\varrho(G)$ are
crisp quasiorders, and \\
\noindent $F(a)\geq\theta(a,y)\odot F(y)$, $G(a)\leq
\theta(a,y)\vartriangleright G(y)$, for any $a,y\in U$.

\noindent(iv) If $n(x)$ is involutive, then $R(F)=\varrho
(n(F))$, and $\varrho(G)=R(n(G))$.

\noindent(v) Let $\odot$ be the minimum t-norm, $n$ an involutive negator, and
$x\vartriangleright y:=$ max$(n(x),y)$. If $\theta$ is a similarity relation,
then $(a,b)\in R(F)\Leftrightarrow F(a)\leq\theta(a,b)$ and $(a,b)\in
\varrho(G)\Leftrightarrow G(a)\geq n(\theta(a,b))$.
\end{proposition}

\begin{proof}
(i) By definition $(a,b)\in R(F)$ implies
$F(a)=\theta(a,b)\odot F(b)\leq\theta(a,b)$ and $(a,b)\in\varrho(G)$ yields
$G(a)=\theta(a,b)\vartriangleright G(b)\geq\theta(a,b)\vartriangleright
0=n(\theta(a,b))$.

\smallskip

(ii) If $\theta$ is reflexive, then $\theta(a,a)=1$ implies $F(a)=$
$\theta(a,a)\odot F(a)$ and $G(a)=$ $\theta(a,a)\vartriangleright G(a)$, i.e.
$(a,a)\in R(F)$ and $(a,a)\in\varrho(G)$ hold for all $a\in U$. Thus $R(F)$ and
$\varrho(G)$ are reflexive.

\smallskip
(iii) Let $\theta$ be a $\odot$-quasiorder. Then $R(F)$, $\varrho(G)$
are reflexive, property (ID) holds, and hence $ F(a)=\overline{\theta}(F)(a)$, $G(a)=\underline{\theta
}(G)(a)$ imply $F(a)=\bigvee\{\theta(x,y)\odot F(y)\mid y\in U\}\geq
\theta(a,y)\odot F(y)$ and $G(a)=\bigwedge\{\theta(x,y)\vartriangleright
G(y)\mid y\in U\}\leq\theta(a,y)\vartriangleright G(y)$, $\forall y\in U$.
Take $a,b,c\in U$ with $(a,b),(b,c)\in R(F)$. Then $F(a)=\theta
(a,b)\odot F(b)$ and $F(b)=\theta(b,c)\odot F(c)$ imply $F(a)=(\theta
(a,b)\odot\theta(b,c))\odot F(c)\leq\theta(a,c)\odot F(c)$, because $\theta$
is $\odot$-transitive. Now $F(a)\geq\ \theta(a,c)\odot F(c)$ yields
$F(a)=\theta(a,c)\odot F(c)$, i.e. $(a,c)\in R(F)$. Thus $R(F)$ is also transitive, hence it is a quasiorder.

\noindent Let $(a,b),(b,c)\in\varrho(G)$. Then $G(a)=\theta(a,b)\vartriangleright
G(b)$ and $G(b)=\theta(b,c)\vartriangleright G(c)$ imply $G(a)=\theta
(a,b)\vartriangleright(\theta(b,c)\vartriangleright G(c))$. If
$\vartriangleright$ is an R-implicator, then $\theta(a,b)\vartriangleright
(\theta(b,c)\vartriangleright G(c))=(\theta(a,b)\odot\theta
(b,c))\vartriangleright G(c)$. If $\vartriangleright$ is an S-implicator
$x\vartriangleright y=n(x)\oplus y$, then $\theta(a,b)\vartriangleright
(\theta(b,c)\vartriangleright G(c))=n(\theta(a,b))\oplus(n(\theta(b,c))\oplus
G(c))=(n(\theta(a,b))\oplus n(\theta(b,c))\oplus
G(c)=n(\theta(a,b)\odot\theta(b,c))\oplus G(c)=(\theta(a,b)\odot
\theta(b,c))\vartriangleright G(c)$. Hence in both cases $G(a)=(\theta
(a,b)\odot\theta(b,c))\vartriangleright G(c)$. Because $\theta$ is $\odot
$-transitive (and $\vartriangleright$ is decreasing in the first variable) we
get $G(a)\geq\theta(a,c)\vartriangleright G(c)$. Then $G(a)\leq\theta
(a,c)\vartriangleright G(c)$ yields $G(a)=\theta(a,c)\vartriangleright G(c)$,
i.e. $(a,c)\in\varrho(G)$. Thus $\varrho(G)$ is a $\odot$-quasiorder.

\smallskip
(iv) Observe, that in this case property (D) holds, i.e.,
$n(\overline{\theta}(f))=\underline{\theta}(n(f))$. This yields
$n(F)=\underline{\theta}(n(f))$, and $(a,b)\in R(F)$ means
$F(a)=\theta(a,b)\odot F(b)$. As $n$ is involutive, this is equivalent to
$n(F(a))=n(\theta(a,b)\odot F(b))$. If $\vartriangleright$ is an ITML
implicator, then $n(\theta(a,b)\odot F(b))=(\theta(a,b)\odot
F(b))\vartriangleright0=\theta(a,b)\vartriangleright(F(b)\vartriangleright0)$
$=\theta(a,b)\vartriangleright n(F(b))$. If $\vartriangleright$ is an
S-implicator, then $n(\theta(a,b)\odot F(b))=n(\theta(a,b))\oplus
n(F(b))=\theta(a,b)\vartriangleright n(F(b))$. Hence in both cases
$n(F(a))=n(\theta(a,b)\odot F(b))\Leftrightarrow n(F)(a)=\theta
(a,b)\vartriangleright n(F)(b)$. The right side means $(a,b)\in\varrho(n(F))$.
Thus we get $(a,b)\in R(F)\Leftrightarrow(a,b)\in\varrho(n(F))$, proving
$R(F)=\varrho(n(F))$.

Let $g=n(h)$ for some $h\in\mathcal{F}(U)$. Then $n(G)=n(\underline{\theta
}(g))=\overline{\theta}(n(g))=\overline{\theta}(h)$, and $G=n(\overline
{\theta}(h))$. Hence $(a,b)\in\varrho(G)\Leftrightarrow$ $(a,b)\in\varrho
(n(\overline{\theta}(h)))\Leftrightarrow(a,b)\in R(\overline{\theta
}(h))=R(n(G))$, and this proves $\varrho(G)=R(n(G))$.

\smallskip
(v) In view of (i) $(a,b)\in R(F)$ yields $F(a)\leq\theta(a,b)$ and
$(a,b)\in\varrho(G)$ implies $G(a)\geq n(\theta(a,b))$. We need only to prove the
converse implications. Let $F(a)\leq\theta(a,b)$. Since $\theta$ is also a $\odot
$-quasiorder, in view of (iii) $F(b)\geq\theta(b,a)\odot F(a)=\ $%
min$(\theta(b,a),F(a))=\ $min$(\theta(a,b),F(a))=F(a)$. Hence $F(a)\leq
\ $min$(\theta(a,b),F(b))$. As $F(a)\geq\theta(a,b)\odot F(b)=\ $%
min$(\theta(a,b),F(b))$ also holds, we get $F(a)=\ $min$(\theta(a,b),F(b))$,
i.e. $(a,b)\in R(F)$.

\noindent Now let $G(a)\geq n(\theta(a,b))$. As $\theta(b,a)=\theta(a,b)$, and
by (iii), $G(b)\leq\theta(b,a)\vartriangleright G(a)=\ $%
max$(n(\theta(a,b)),G(a))=G(a)$, we get $G(a)\geq\ $max$(n(\theta(a,b),G(b))$.
Since $G(a)\leq\theta(a,b)\vartriangleright G(b)=\ $max$(n(\theta(a,b)),G(b))$
also holds, we obtain $G(a)=\ $max$(n(\theta(a,b)),G(b))=\theta
(a,b)\vartriangleright G(b)$, i.e. $(a,b)\in\varrho(G)$.
\end{proof}

\begin{corollary}
\label{cor:cor1}
If the conditions in Proposition \ref{prop:prop1}(v) are
satisfied, then $(a,b)\notin R(F)\Leftrightarrow F(a)>\theta(a,b)$ and
$(a,b)\notin\varrho(G)\Leftrightarrow G(a)<n(\theta(a,b))$.
\end{corollary}

\begin{proposition}
\label{prop:prop2}
Let $\theta$ be a $\odot$-quasiorder and
$F=\overline{\theta}(f)$, $G=\underline{\theta}(g)$, for some $f,g\in
\mathcal{F}(U)$, and $a,b\in U$. The following hold true:

\noindent(i) If $(a,b)\in R(F)$, then for any $h\in\mathcal{F}(U)$ with $h\leq
F$, $h(b)=F(b)$ implies $\overline{\theta}(h)(a)=F(a)$.

\noindent(ii) If $(a,b)\in\varrho(G)$, then for any $h\in\mathcal{F}(U)$ with
$h\geq G$, $h(b)=G(b)$ implies $\underline{\theta}(h)(a)=G(a)$.
\end{proposition}

\begin{proof} (i) By definition, we have $F(a)=$ $\theta(a,b)\odot
F(b)$. Hence we get:

\begin{center}
$\overline{\theta}(h)(a)=\bigvee\{\theta(a,y)\odot h(y)\mid y\in U\}\geq\ $

$\theta(a,b)\odot h(b)=\ \theta(a,b)\odot F(b)=F(a)$
\end{center}

\noindent On the other hand, $\overline{\theta}(h)\leq\overline{\theta}(F)=F$
implies $\overline{\theta}(h)(a)\leq F(a)$. Thus we obtain $\overline{\theta
}(h)(a)=F(a)$.

\noindent(ii) Now, analogously we have $G(a)=$ $\theta(a,b)\vartriangleright
G(b)$. Therefore, we get:

\begin{center}
$\underline{\theta}(h)(a)=\bigwedge\{\theta(a,y)\vartriangleright h(y)\mid
y\in U\}\leq\ $

$\theta(a,b)\vartriangleright h(b)=\theta(a,b)\vartriangleright G(b)=G(a)$.
\end{center}

\noindent Now $\underline{\theta}(h)\geq\underline{\theta}(G)=G$
yields $\underline{\theta}(h)(a)\geq G(a)$, whence $\underline{\theta
}(h)(a)=G(a)$. \end{proof}

\section{The equivalences induced by the quasiorders $R(F)$\textbf{ and
}$\varrho(G)$}
\label{sec:eqbyqo}

In this section we assume that conditions in (ID) are
satisfied, i.e. that $\odot$ is a left-continuous t-norm, $\vartriangleright
$ is the R-implicator induced by it and $n(x)=x\vartriangleright0$, or $n$ is an involutive negator, $\oplus$ is the
t-conorm $n$-dual to $\odot$ and $\vartriangleright$ is the S-implicator
defined by them. We also suppose that $(U,\theta)$ is an approximation space with a $\odot$-similarity relation $\theta$ and
$F=\overline{\theta}(f)$, $G=\underline{\theta}(g)$, for some $f,g\in
\mathcal{F}(U)$. Now, by Proposition \ref{prop:prop1}(iii), $R(F)$, $\varrho
(G)\subseteq U\times U$ are (crisp) quasiorders. It is known that
for any quasiorder $q\subseteq U\times U$ the relation $\varepsilon
_{q}:=q\cap q^{-1}$ is an equivalence and $q$ induces a \emph{natural
partial order }$\leq_{q}$ on the factor-set $U/\varepsilon_{q}$ as follows: for
any equivalence classes $A,B\in U/\varepsilon_{q}$ we say that $A\leq_{q}B$, whenever there exists $a\in A$ and $b\in B$ with $(a,b)\in q$. This is equivalent to the fact that $(x,y)\in q$ holds for all $x\in A$ and
$y\in B$.

\noindent Thus we can introduce two equivalence relations $E(F)$
and $\varepsilon (G)$ as follows:%
\[
E(F):=R(F)\cap R(F)^{-1}\text{and }\varepsilon%
(G):=\varrho(G)\cap\varrho(G)^{-1}\text{.}%
\]

\noindent The corresponding (natural) partial orders on the factor-sets
$U/E(F)$ and $U/\varepsilon (G)$ can be defined as follows:

For any $E_{1},E_{2}\in U/E(F)$, we have $E_{1}\leq_{R(F)}%
E_{2}\Leftrightarrow(a_{1},a_{2})\in R(F)$ for some $a_{1}\in E_{1}$ and
$a_{2}\in E_{2}$, and for any $\mathcal{E}_{1},\mathcal{E}_{2}\in
U/\varepsilon (G)$ we have $\mathcal{E}_{1}\leq_{\varrho(G)}%
\mathcal{E}_{2}\Leftrightarrow(b_{1},b_{2})\in\varrho(G)$ for some $b_{1}%
\in\mathcal{E}_{1}$ and $b_{2}\in\mathcal{E}_{2}$.

\noindent $E$ is called a \emph{maximal }$E(F)$\emph{\ class}, if it is a
maximal element of the poset $(U/E(F),\leq_{R(F)})$ and $\mathcal{E}$ is a \emph{maximal }$\varepsilon (G)$\emph{\ class} if it is 
maximal in $(U/\varepsilon (G),\leq_{\varrho(G)})$. The
$E(F)$ and $\varepsilon (G)$ class of an $a\in U$ is
denoted by $[a]_{E(F)}$ and $[a]_{\varepsilon (G)} $, respectively.
In this section we prove several properties of these classes used to characterize
pairs of fuzzy sets which together form fuzzy rough sets. 
\medskip

\begin{lemma}
\label{lem:assertions}
The following assertions hold true:

\noindent(i) If $E\subseteq U$ is an $E(F)$ class, then
$F(a)=F(b)\leq\theta(a,b)$, for all $a,b\in E$;

\noindent(ii) If $\mathcal{E}\subseteq U$ is an $\varepsilon (G)$
class, then $G(a)=G(b)\geq n(\theta(a,b))$, for all $a,b\in\mathcal{E}$;

\noindent(iii) If $E\subseteq U$ is a maximal $E(F)$ class, then
$\theta(a,z)\odot F(z)<F(a)=F(b)\leq\theta(a,b)$ and $\theta(a,z)<\theta
(a,b)$, for all $a,b\in E$ and $z\notin E$;

\noindent(iv) If $\mathcal{E}\subseteq U$ is a maximal $\varepsilon (G)$ class, then $n(\theta(a,b))\leq G(a)=G(b)<\theta(a,z)\vartriangleright
G(z)$ and $\theta(a,z)<\theta(a,b)$, for all $a,b\in\mathcal{E}$ and
$z\notin\mathcal{E}$;

\noindent(v) Assume that $n(x)=x\vartriangleright0$ is involutive. Then the
$E(F)$ classes and the $\varepsilon (n(F))$ classes are the
same, and $E\subseteq U$ is a maximal $E(F)$ class if and only if it
is also a maximal $\varepsilon (n(F))$ class.
\end{lemma}

\begin{proof} (i) If $E\subseteq U$ is an $E(F)$ class,
then $(a,b),(b,a)\in R(F)$ holds for any $a,b\in E$. Therefore, $F(a)=$
$\theta(a,b)\odot F(b)\leq\theta(a,b)$, $F(b)$ and $F(b)=\theta(b,a)\odot
F(a)\leq F(a)$. Thus we obtain $F(b)=F(a)\leq\theta(a,b)$.

(ii) If $\mathcal{E}\subseteq U$ is an $\varepsilon (G)$
class, then $(a,b),(b,a)\in\varrho(G)$ imply $G(a)=\theta(a,b)\vartriangleright
G(b)\geq1\vartriangleright G(b)=G(b)$ and $G(b)\geq\theta
(b,a)\vartriangleright G(a)\geq1\vartriangleright G(a)=G(a)$, for any
$a,b\in\mathcal{E}$. Hence $G(a)=G(b)$. By Proposition \ref{prop:prop1}(i) we obtain
$G(a)=G(b)\geq n(\theta(a,b))$, for all $a,b\in\mathcal{E}$.

(iii) Let $E$ be a maximal $E(F)$ class. Then $F(a)=\theta
(a,b)\odot F(b)$, $F(b)=\theta(a,b)\odot F(a)$, and in view of (i),
$F(a)=F(b)\leq\theta(a,b)$. We also have $F(a)\geq\ \theta(a,z)\odot F(z)$, according to Proposition \ref{prop:prop1}(iii). As $E\nleq\lbrack z]_{E(F)}$
implies $(a,z)\notin R(F)$ for each $z\notin E$, we obtain $F(a)>\theta
(a,z)\odot F(z)$, for all $z\notin E$. Now, suppose that $\theta
(a,c)\geq\theta(a,b)$, for some $c\notin E$. Then $F(c)\geq\ \theta(c,a)\odot
F(a)=\theta(a,c)\odot F(b)\geq\theta(a,b)\odot F(b)=F(a)$. This further yields
$F(a)>\theta(a,c)\odot F(c)\geq\theta(a,b)\odot F(a)=F(b)$, a contradiction.
Thus $\theta(a,z)<\theta(a,b)$, for each $z\notin E$.

(iv) If $\mathcal{E}\subseteq U$ is a maximal $\varepsilon
(G)$ class, then $G(a)=\theta(a,b)\vartriangleright G(b)$, $G(b)=\theta
(a,b)\vartriangleright G(a)$, and in view of (i), $n(\theta(a,b))\leq
G(a)=G(b)$, for all $a,b\in\mathcal{E}$ and $\mathcal{E}\nleq\lbrack
z]_{\varepsilon (G)}$, for any $z\notin\mathcal{E}$. Now, by Proposition \ref{prop:prop1}(iii), $G(a)\leq\theta
(a,z)\vartriangleright G(z)$, hence $(a,z)\notin\varrho(G)$ yields $G(a)<\theta
(a,z)\vartriangleright G(z)$. By way of contradiction, assume
$\theta(a,c)\geq\theta(a,b)$, for some $c\notin E$. Then $G(c)\leq
\theta(a,c)\vartriangleright G(a)\leq\theta(a,b)\vartriangleright G(b)=G(a)$.
This further yields $G(a)<\theta(a,c)\vartriangleright G(c)\leq\theta
(a,b)\vartriangleright G(a)=G(b)$, a contradiction again.

(v) If $n(x)=x\vartriangleright0$ is involutive, then property (D)
means that $n(F)=n(\overline{\theta}(f))=\underline{\theta}(n(f))$. Hence, relation $\varrho(n(F))=R(F)$ is well defined, and $E(F)=R(F)\cap
R(F)^{-1}=\varrho(n(F))\cap\varrho(n(F))^{-1}=\varepsilon (n(F))$. Thus
the equivalence classes of $E(F)$ and $\varepsilon (n(F))$
coincide. $R(F)=\varrho(n(F))$ also yields $\leq_{R(F)}=\leq_{\varrho(n(F))}$, i.e.
the posets $(U/E(F),\leq_{R(F)})$ and $(U/\varepsilon
(n(F)),\leq_{\varrho(n(F))})$ are the same. Therefore, the maximal $E%
(F)$ and $\varepsilon (n(F))$ classes coincide.
\end{proof}

\begin{corollary}
\label{cor:assertions}
Let $E$ be an $E(F)$\ class and
$\mathcal{E}$ be an $\varepsilon (G)$ class such that $E\cap
\mathcal{E}\neq\emptyset$. Then the following assertions hold:

\noindent(i) If $E\subseteq U$ is a maximal $E(F)$ class and
$\mathcal{E}\subseteq U$ is a maximal $\varepsilon (G)$ class, then
$E\subseteq\mathcal{E}$ or $\mathcal{E}\subseteq E$ holds.

\noindent(ii) If $\theta$ is a similarity relation, then $(x,y)\in R(F)$ or
$(y,x)\in\varrho(G)$ holds for all $x\in E$ and $y\in\mathcal{E}$.
\end{corollary}

\begin{proof} Let $a\in E\cap\mathcal{E}$. (i) Assume that neither
$E\subseteq\mathcal{E}$ nor $\mathcal{E}\subseteq E$ hold. Then there exist
elements $b\in E\setminus\mathcal{E}$, $c\in\mathcal{E}\setminus E$. As
$a,b\in E$ but $c\notin E$, in view of Lemma \ref{lem:assertions}(iii) we have $\theta
(a,c)<\theta(a,b)$. Similarly, $a,c\in\mathcal{E}$ and $b\notin\mathcal{E}$
imply $\theta(a,b)<\theta(a,c)$, a contradiction to the previous result.

(ii) If $\mathcal{E}\subseteq E$ or $E\subseteq\mathcal{E}$ then (ii) is
clearly satisfied. Hence we may assume $\mathcal{E}\setminus E\neq\emptyset$
and $E\setminus\mathcal{E}\neq\emptyset$. Suppose that there exist
$x\in E$ and $y\in\mathcal{E}$ with $(x,y)\notin R(F)$. We claim that
$(y,x)\in\varrho(G)$. Assume by contradiction $(y,x)\notin\varrho(G)$. Since $x,a\in
E$, $y\notin E$ and $y,a\in\mathcal{E}$, $x\notin\mathcal{E}$, in view of
Lemma \ref{lem:assertions}(iii) and (iv) we get $\theta(x,y)<\theta(x,a)$ and $\theta
(x,y)=\theta(y,x)<\theta(y,a)=\theta(a,y)$. Thus we obtain $\theta
(x,y)<\ $min$(\theta(x,a),\theta(a,y))\leq\theta(x,y)$, a contradiction. This
proves $(y,x)\in\varrho(G)$. \end{proof}

\begin{proposition}
\label{prop:twodiffeqclass}
(i) If $E_{1},E_{2}$ are different
$E(F)$ classes with $E_{1}\leq_{R(F)}E_{2}$, then for any $a_{1}\in
E_{1}$ and $a_{2}\in E_{2}$ we have $F(a_{1})<F(a_{2})$.

\noindent(ii) If $\mathcal{E}_{1},\mathcal{E}_{2}$ are different
$\varepsilon (G)$ classes with $\mathcal{E}_{1}\leq_{\varrho
(G)}\mathcal{E}_{2}$ then for any $b_{1}\in\mathcal{E}_{1}$ and $b_{2}%
\in\mathcal{E}_{2}$ we have $G(b_{1})>G(b_{2})$.
\end{proposition}

\begin{proof} (i) Assume $E_{1}\leq_{R(F)}E_{2}$. Then for any
$a_{1}\in E_{1}$ and $a_{2}\in E_{2}$ we have $(a_{1},a_{2})\in R(F)$, i.e.
$F(a_{1})=\theta(a_{1},a_{2})\odot F(a_{2})\leq F(a_{2})$. Observe that
$F(a_{2})\neq F(a_{1})$. Indeed, $F(a_{2})=F(a_{1})$ would imply
$F(a_{2})=\theta(a_{1},a_{2})\odot F(a_{1})=\theta(a_{2},a_{1})\odot F(a_{1}%
)$, i.e. $(a_{2},a_{1})\in R(F)$, which means $E_{2}\leq_{R(F)}E_{1}$. As
$\leq_{R(F)}$ is a partial order, this would yield $E_{1}=E_{2}$, a
contradiction. Thus we deduce $F(a_{1})<F(a_{2})$.

\noindent(ii) Let $\mathcal{E}_{1}\leq_{\varrho(G)}\mathcal{E}_{2}$. Then for
any $b_{1}\in\mathcal{E}_{1}$, $b_{2}\in\mathcal{E}_{2}$ we have
$(b_{1},b_{2})\in\varrho(G)$, which gives $G(b_{1})=\theta(b_{1},b_{2}%
)\vartriangleright G(b_{2})\geq G(b_{2}).$ We claim
$G(b_{1})>G(b_{2})$. Indeed, $G(b_{2})=G(b_{1})$ would imply $G(b_{2}%
)=\theta(b_{1},b_{2})\vartriangleright G(b_{1})=\theta(b_{2},b_{1}%
)\vartriangleright G(b_{1})$, i.e. $(b_{2},b_{1})\in\varrho(G)$, which would yield $\mathcal{E}_{1}=\mathcal{E}_{2}$, a contradiction. \end{proof}

Clearly, if each chain in the posets $(U/E(F),\leq_{R(F)})$ and
$(U/\varepsilon (G),\leq_{\varrho(G)})$ is finite, then any element of
them is less than or equal to a maximal element in the corresponding poset. By
using this observation we deduce

\begin{corollary}
\label{cor:maxeqclexists}
Assume\textbf{ }that the relation $\theta$
and the fuzzy sets $f,g\in\mathcal{F}(U)$ have a finite range, and let
$F=\overline{\theta}(f)$, $G=\underline{\theta}(g)$. Then for any
$E(F)$ class $E$, there exists a maximal $E(F)$ class
$E_{M}$ such that $E\leq_{R(F)}E_{M}$, and for any $\varepsilon (G)$
class $\mathcal{E}$, there is a maximal $\varepsilon (G)$ class
$\mathcal{E}_{M}$ with $\mathcal{E}\leq_{\varrho(G)}\mathcal{E}_{M}$.
\end{corollary}

\begin{proof} If the above conditions hold, then the fuzzy sets $F$
and $G$ also have a finite range. Now let $\{E_{i}\mid i\in I\}$ be an arbitrary
(nonempty) chain of $E(F)$ classes. In view of Proposition \ref{prop:twodiffeqclass}, for
any $a_{i}\in E_{i}$, $i\in I$, the values $\{F(a_{i})\mid i\in I\}$ also form
a chain, and for $E_{i}\leq_{R(F)}E_{j}$, $E_{i}\neq E_{j}$ we have
$F(a_{i})<F(a_{j})$, and vice versa. This means that the chains $\{E_{i}\mid i\in I\}$ and
$\{F(a_{i})\mid i\in I\}$ are order-isomorphic. Since $F$ has a finite range,
the chain $\{F(a_{i})\mid i\in I\}$ has a finite length. Hence the chain
$\{E_{i}\mid i\in I\}$ is also finite. As every chain in the poset
$(U/E(F),\leq_{R(F)})$ is finite, any element $E$ of it is less than or
equal to a maximal element $E_{M}$ of it, i.e. $E\leq_{R(F)}E_{M}$. 
The second statement is proved analogously. \end{proof}

\noindent The importance of maximal classes in this case is shown by the following:

\begin{proposition}
\label{prop:hufuhvgv}
Suppose that $\theta$ and the fuzzy sets
$f,g\in\mathcal{F}(U)$ have a finite range, and let $F=\overline{\theta}(f)$,
$G=\underline{\theta}(g)$. Then the following assertions hold:

\noindent(i) If $h\leq F$ for some $h\in\mathcal{F}(U)$ and for any maximal
$E(F)$ class $E_{M}$ there exists an element $u\in E_{M}$ with
$h(u)=F(u)$, then $\overline{\theta}(h)=F$.

\noindent(ii) If $h\geq G$ for some $h\in\mathcal{F}(U)$ and for any maximal
$\varepsilon (G)$ class $\mathcal{E}_{M}$ there exists an element
$v\in\mathcal{E}_{M}$ with $h(v)=G(v)$, then $\underline{\theta}(h)=G$.
\end{proposition}

\begin{proof} (i) Let $x\in U$ be arbitrary. As $\theta$ and $f$
have finite ranges, in view of Corollary \ref{cor:maxeqclexists}, there exists a maximal
$E(F)$ class $E_{M}$ such that $[x]_{E(F)}\leq_{R(F)}E_{M}$. Then
$(x,y)\in R(F)$ for all $y\in E_{M}$. By assumption, there exists an element
$u\in E_{M}$ with $h(u)=F(u)$. Since $h\leq F$ and $(x,u)\in R(F)$, in view of
Proposition \ref{prop:prop2}(i) we obtain $\overline{\theta}(h)(x)=F(x)$. This proves
$\overline{\theta}(h)=F$.

\noindent(ii) is proved dually, by using Corollary \ref{cor:maxeqclexists} and Proposition \ref{prop:prop2}(ii).
\end{proof}

\begin{proposition}
\label{prop:famaxfy}
Suppose that the relation $\theta$ and the
fuzzy sets $f,g\in\mathcal{F}(U)$ have a finite range, and let $F=\overline
{\theta}(f)$, $G=\underline{\theta}(g)$.

\noindent(i) If $E$ is a maximal $E(F)$ class, then for any $a\in E$
we have

$F(a)=\ $max$\{f(y)\mid y\in E\}$.

\noindent(ii) If $\mathcal{E}$ is a maximal $\varepsilon (G)$ class,
then for any $a\in\mathcal{E}$ we have

$G(a)=\ $min$\{g(y)\mid y\in\mathcal{E}\}$.
\end{proposition}

\begin{proof} (i) By definition, $F(a)=\overline{\theta
}(f)(a)=\bigvee\{\theta(a,y)\odot f(y)\mid y\in U\}$. If $y\notin E$, then
$(a,y)\notin R(F)$, because $E$ is a maximal $E(F)$ class. This means
that $F(a)=\theta(a,y)\odot F(y)$ is not possible, and hence $F(a)>\theta
(a,y)\odot F(y)$, according to Proposition \ref{prop:prop1}(iii). Since $f\leq
\overline{\theta}(f)=F$, we obtain $F(a)>\theta(a,y)\odot f(y)$, for all $y\in
U\setminus E$. As $\theta$ and $f$ are of a finite range, the set
$\{\theta(a,y)\odot f(y)\mid y\in U\setminus E\}$ has finitely many different
elements, and hence $\bigvee\{\theta(a,y)\odot f(y)\mid y\in U\setminus
E\}<F(a)$. This implies \medskip

$F(a)=(\bigvee\{\theta(a,y)\odot f(y)\mid y\in E\})\vee\left(  \bigvee
\{\theta(a,y)\odot f(y)\mid y\in U\setminus E\}\right)  =$

$\bigvee\{\theta(a,y)\odot f(y)\mid y\in E\}$.\medskip

\noindent If $y\in E$, then $F(y)=F(a)$. As $\theta(a,y)\odot
f(y)\leq f(y)\leq F(y)=F(a)$, we obtain:

$F(a)=$ $\bigvee\{\theta(a,y)\odot f(y)\mid y\in E\}\leq\bigvee\{f(y)\mid y\in
E\}\leq F(a)$.

\noindent This implies $F(a)=$ $\bigvee\{f(y)\mid y\in E\}$. Because $f$ has a
finite range, the set $\{f(y)\mid y\in E\}$ is finite, and hence we can write $F(a)=\ $max$\{f(y)\mid y\in E\}$.

\noindent(ii) By definition $G(a)=\underline{\theta}(g)(a)=\bigwedge
\{\theta(a,y)\vartriangleright g(y)\mid y\in U\}$. If $y\notin\mathcal{E}$,
then $(a,y)\notin\varrho(G)$, because $\mathcal{E}$ is a maximal
$\varepsilon (G)$ class, and hence $G(a)\neq\theta
(a,y)\vartriangleright G(y)$. Thus we have $G(a)<\theta(a,y)\vartriangleright
G(y)$, according to Proposition \ref{prop:prop1}(iii). Since $G=\underline{\theta}(g)\leq
g$, we obtain $G(a)<\theta(a,y)\vartriangleright g(y)$, for all $y\in
U\setminus\mathcal{E}$. As $\theta$ and $g$ are of a finite range, the set
$\{\theta(a,y)\vartriangleright g(y)\mid y\in U\setminus E\}$ is finite,
whence we get $G(a)<\bigwedge\{\theta(a,y)\odot g(y)\mid y\in U\setminus E\}$.
This yields

$G(a)=(\bigwedge\{\theta(a,y)\vartriangleright g(y)\mid y\in\mathcal{E}%
\})\wedge(\bigwedge\{\theta(a,y)\vartriangleright g(y)\mid y\in U\setminus
\mathcal{E}\})=$

$\bigwedge\{\theta(a,y)\vartriangleright g(y)\mid y\in\mathcal{E}\}$.

\noindent If $y\in\mathcal{E}$, then $G(y)=G(a)$. Since $\theta
(a,y)\vartriangleright g(y)\geq1\vartriangleright g(y)=g(y)\geq G(y)=G(a)$, we obtain

$G(a)=\bigwedge\{\theta(a,y)\vartriangleright g(y)\mid y\in\mathcal{E}%
\}\geq\bigwedge\{g(y)\mid y\in\mathcal{E}\}\geq G(a)$, $\ $

\noindent and this implies $G(a)=\bigwedge\{g(y)\mid y\in\mathcal{E}\}$. Since
$\{g(y)\mid y\in\mathcal{E}\}$ is a finite set, we can write: $G(a)=\ $%
min$\{g(y)\mid y\in\mathcal{E}\}$. \end{proof}

The following corollary is immediate:

\begin{corollary}
\label{cor:fafagaga}
Assume that $\theta$ and $f,g\in
\mathcal{F}(U)$ are of a finite range, and let $a\in U$ and $F=\overline
{\theta}(f)$, $G=\underline{\theta}(g)$.

\noindent(i) If $\{a\}$ is a maximal $E(F)$ class, then $F(a)=f(a)$.

\noindent(ii) If $\{a\}$ is a maximal $\varepsilon (G)$ class, then
$G(a)=g(a)$.
\end{corollary}

\begin{corollary}
\label{cor:fugufvgv}
Assume that $\theta$ and $f\in\mathcal{F}(U)$
are of a finite range, $F=\overline{\theta}(f)$, $G=\underline{\theta}(f)$,
and let $E$ be a maximal $E(F)$ class and $\mathcal{E}$ a maximal
$\varepsilon (G)$ class. 
\noindent(i) If every $\{x\}\subseteq E$ is a maximal $\varepsilon
(G)$ class, then there exists a $u\in E$ with $F(u)=G(u)$.

\noindent(ii) If every $\{x\}\subseteq\mathcal{E}$ is a maximal
$E(F)$ class, then there exists a $v\in\mathcal{E}$ with $F(v)=G(v)$.
\end{corollary}

\begin{proof} (i) In view of Proposition \ref{prop:famaxfy}(i), for each $x\in E$
we have $F(x)=\ $max$\{f(y)\mid y\in E\}$, i.e. $F(x)=f(u)$, for some $u\in
E$. As $\{u\}$ is a maximal $\varepsilon (G)$ class, by applying
Corollary \ref{cor:fafagaga}(ii) with $g:=f$ we get $G(u)=f(u)$. Hence $F(u)=G(u)$.
(ii) is proved dually. \end{proof}

\begin{example}
\label{ex:factor}
Let us consider the similarity relation $\theta $, a fuzzy set $h$ and its approximations $F=\overline{\theta}(h)$, $G=\underline{\theta}(h)$ given on Figure \ref{fig:theta2} and Table \ref{tab:fuzzy2}.

\begin{figure}[H]
    \centering
    \begin{tikzpicture}
    \draw (-2, 2.3) node {$a$};
    \draw (-2, -0.3) node {$b$};
    \draw (1, 2.3) node {$c$};
    \draw (1, -0.3) node {$d$};
    \draw (3, 2.3) node {$e$};
    \draw (3, -0.3) node {$f$};

    \draw (-2, 0) -- (-2, 2);
    \draw[dashed] (-2, 0) -- (1, 2) -- (-2, 2);
    \draw[dashed] (-2, 0) -- (1, 0) -- (-2, 2);
    \draw[dashed] (3, 0) -- (3, 2);
    \draw[dashed] (1, 0) -- (1, 2);
    
    \draw[fill=white] (-2, 0) circle [radius=2pt];
    \draw[fill=white] (3, 0) circle [radius=2pt];
    \draw[fill=white] (1, 2) circle [radius=2pt];
    \draw[fill=white] (1, 0) circle [radius=2pt];
    \draw[fill=white] (-2, 2) circle [radius=2pt];
    \draw[fill=white] (3, 2) circle [radius=2pt];

    \draw (-2.3, 1) node {$1$};
    \draw (-0.5, -0.3) node {$0.25$};
    \draw (-0.5, 2.3) node {$0.75$};
    \draw (1.5, 1) node {$0.25$};
    \draw (-1.5, 0.75) node {$0.75$};
    \draw (0.5, 0.75) node {$0.25$};
    \draw (3.4, 1) node {$0.5$};
    \end{tikzpicture}
    \caption{The fuzzy similarity relation $\theta$ of Example \ref{ex:factor}}
    \label{fig:theta2}
\end{figure}
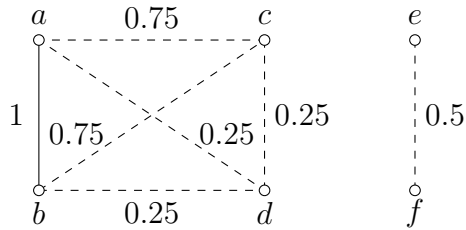

\begin{table}[H]
    \centering
    \begin{tabular}{|P{0.75cm}|P{0.75cm}|P{0.75cm}|P{0.75cm}|P{0.75cm}|P{0.75cm}|P{0.75cm}|}
        \hline
         $u$ & $a$ & $b$ & $c$ & $d$ & $e$ & $f$  \\ \hline
         $h(u)$ & 0 & 1 & 0.25 & 0.5 & 0.5 & 0.75  \\ \hline
         $F(u)$ & 1 & 1 & 0.75 & 0.5 & 0.5 & 0.75 \\ \hline
         $G(u)$ & 0 & 0 & 0.25 & 0.5 & 0.5 & 0.5 \\ \hline
    \end{tabular}
    \caption{The fuzzy set $h$ of Example \ref{ex:factor}}
    \label{tab:fuzzy2}
\end{table}

\noindent The quasiorders $R(F)$ and
$\varrho(G)$, their $E(F)$ and $\varepsilon (G)$ equivalence classes and the partial orders induced on the factor-sets are given in Figure \ref{fig:factor2}.

\begin{figure}[H]
    \centering
    \begin{tikzpicture}
    \draw (-2, 2.3) node {$a$};
    \draw (-2, -0.3) node {$b$};
    \draw (0, 2.3) node {$c$};
    \draw (0, -0.3) node {$d$};
    \draw (2, 2.3) node {$e$};
    \draw (2, -0.3) node {$f$};

    \draw (0, -1) node {$\varrho(G)$};

    \draw (4, 2.3) node {$a$};
    \draw (4, -0.3) node {$b$};
    \draw (6, 2.3) node {$c$};
    \draw (6, -0.3) node {$d$};
    \draw (8, 2.3) node {$e$};
    \draw (8, -0.3) node {$f$};

    \draw (6, -1) node {$R(F)$};
    
    \draw[latex-latex] (-2, 0.1) -- (-2, 1.9);
    \draw[-latex] (0, 2) -- (-1.9, 2);
    \draw[-latex] (0, 2) -- (-1.9, 0.1);
    \draw[latex-latex] (2, 0.1) -- (2, 1.9);

    \draw[-latex] (0, 0) -- (-1.9, 1.9);
    \draw[-latex] (0, 0) -- (-1.9, 0);
    \draw[-latex] (0, 0) -- (0, 1.9);

    \draw[latex-latex] (4, 0.1) -- (4, 1.9);
    \draw[-latex] (6, 2) -- (4.1, 2);
    \draw[-latex] (6, 2) -- (4.1, 0.1);
    \draw[latex-] (8, 0.1) -- (8, 2);

    \draw[-latex] (6, 0) -- (4.1, 1.9);
    \draw[-latex] (6, 0) -- (4.1, 0);
    \draw[-latex] (6, 0) -- (6, 1.9);
    
    \draw[fill=white] (-2, 0) circle [radius=2pt];
    \draw[fill=white] (2, 0) circle [radius=2pt];
    \draw[fill=white] (0, 0) circle [radius=2pt];
    \draw[fill=white] (0, 2) circle [radius=2pt];
    \draw[fill=white] (-2, 2) circle [radius=2pt];
    \draw[fill=white] (2, 2) circle [radius=2pt];

    \draw[fill=white] (4, 0) circle [radius=2pt];
    \draw[fill=white] (8, 0) circle [radius=2pt];
    \draw[fill=white] (6, 0) circle [radius=2pt];
    \draw[fill=white] (6, 2) circle [radius=2pt];
    \draw[fill=white] (4, 2) circle [radius=2pt];
    \draw[fill=white] (8, 2) circle [radius=2pt];

    \draw[rounded corners][color=blue] (-2.5, 2.5) rectangle (-1.5, -0.5) {};
    \draw[rounded corners][color=blue] (0.5, 1.5) rectangle (-0.5, 2.5) {};
    \draw[rounded corners][color=blue] (0.5, -0.5) rectangle (-0.5, 0.5) {};
    \draw[rounded corners][color=blue] (2.5, 2.5) rectangle (1.5, -0.5) {};

    \draw[rounded corners][color=red] (3.5, 2.5) rectangle (4.5, -0.5) {};
    \draw[rounded corners][color=red] (6.5, 1.5) rectangle (5.5, 2.5) {};
    \draw[rounded corners][color=red] (6.5, -0.5) rectangle (5.5, 0.5) {};
    \draw[rounded corners][color=red] (8.5, 1.5) rectangle (7.5, 2.5) {};
    \draw[rounded corners][color=red] (8.5, -0.5) rectangle (7.5, 0.5) {};

    \draw[color=blue] (-2.8, 1) node {$\mathcal{E}_1$};
    \draw[color=blue] (0.8, 2) node {$\mathcal{E}_2$};
    \draw[color=blue] (0.8, 0) node {$\mathcal{E}_3$};
    \draw[color=blue] (1.25, 1) node {$\mathcal{E}_4$};

    \draw[color=red] (3.2, 1) node {$E_1$};
    \draw[color=red] (6.8, 2) node {$E_2$};
    \draw[color=red] (6.8, 0) node {$E_3$};
    \draw[color=red] (8.8, 2) node {$E_4$};
    \draw[color=red] (8.8, 0) node {$E_5$};
    
    \draw (-1, -2) -- (-1, -3) -- (-1, -4);

    \draw[fill=white] (-1, -2) circle [radius=2pt];
    \draw[fill=white] (-1, -3) circle [radius=2pt];
    \draw[fill=white] (-1, -4) circle [radius=2pt];
    \draw[fill=white] (1, -3) circle [radius=2pt];

    \draw [color=blue] (-1.3, -2) node {$\mathcal{E}_1$};
    \draw [color=blue] (-1.3, -3) node {$\mathcal{E}_2$};
    \draw [color=blue] (-1.3, -4) node {$\mathcal{E}_3$};
    \draw [color=blue] (1.3, -3) node { $\mathcal{E}_4$};

    \draw (5, -2) -- (5, -3) -- (5, -4);
    \draw (7, -2.5) -- (7, -3.5);

    \draw[fill=white] (5, -2) circle [radius=2pt];
    \draw[fill=white] (5, -3) circle [radius=2pt];
    \draw[fill=white] (5, -4) circle [radius=2pt];
    \draw[fill=white] (7, -2.5) circle [radius=2pt];
    \draw[fill=white] (7, -3.5) circle [radius=2pt];

    \draw[color=red] (4.65, -2) node {$E_1$};
    \draw[color=red] (4.65, -3) node {$E_2$};
    \draw[color=red] (4.65, -4) node {$E_3$};
    \draw[color=red] (7.35, -3.5) node {$E_4$};
    \draw[color=red] (7.35, -2.5) node {$E_5$};
    
    \end{tikzpicture} 
    \caption{The quasiorders, the factor-sets and their Hasse-diagrams for Example \ref{ex:factor}}
    \label{fig:factor2}
\end{figure}
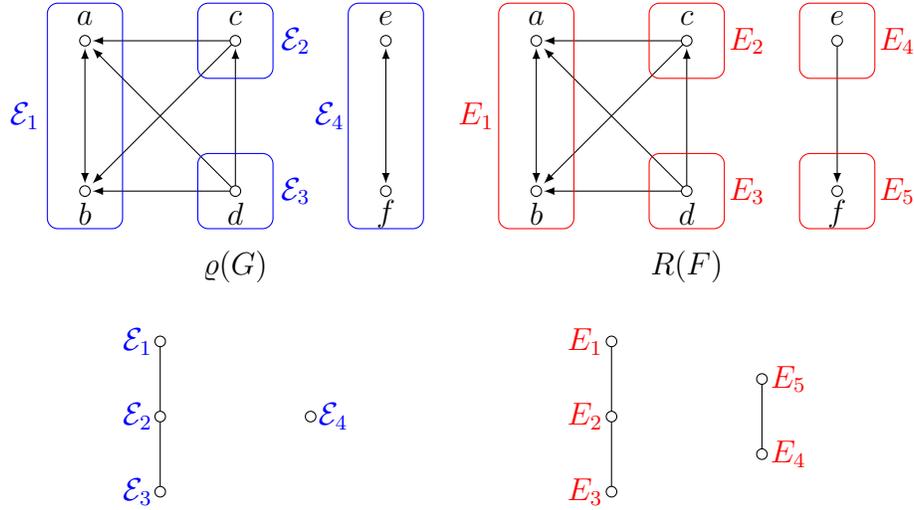
\noindent  Loops are not drawn for any relation. As $\theta$ is symmetric, its edges are undirected, and those with $\theta(x,y)=0$ are not shown either. The maximal $\varepsilon (G)$ classes are $ \mathcal{E}_1$, $ \mathcal{E}_4$, and the maximal $E(F)$ classes are $E_1 $ and $E_5 $. 
In all our examples the approximations are defined by a min t-norm and the KD-implicator max$(1-x,y)$. Clearly, all statements in \ref{lem:assertions}, \ref{cor:assertions}, \ref{cor:fugufvgv} hold.
\end{example}

\section{A characterization of fuzzy rough sets}
\label{sec:frs}

In case of an equivalence $\varrho\subseteq U\times U$, the sets $\{X\subseteq
U\mid X_{\varrho}=X\}$ and $\{X\subseteq U\mid X^{\varrho}=X\}$ coincide and
their members are called $\varrho$\emph{-definable} subsets of $U$. They can
be described as those subsets of $U$ which are the union of some
$\varrho$-equivalence classes, and their set is denoted by Def$(U,\varrho
)$. The rough sets induced by an equivalence relation $\varrho\subseteq
U\times U$ can be characterized by using the \emph{set of} its \emph{singletons}
$S=\{s\in U\mid\varrho(s)=\{s\}\}$, as follows: 

\medskip
\noindent $(A,B)$ is rough set of $\varrho$, if and only if $(A,B)\in\ $Def$(U,\varrho)\times\ $%
Def$(U,\varrho)$, $A\subseteq B$ and $A\cap S=B\cap S$ (see e.g. \cite{JPR}).

\medskip

\noindent In this section we will derive an analogous characterisation for the
fuzzy rough sets with finite ranges and satisfying conditions in (ID). For a fuzzy
approximation space $(U,\theta)$ we will introduce the notations:%
\[
\text{Fix}(\underline{\theta})=\{f\in\mathcal{F}(U)\mid\underline{\theta
}(f)=f\}\text{, Fix}(\overline{\theta})=\{f\in\mathcal{F}(U)\mid
\overline{\theta}(f)=f\}\text{.}%
\]

\noindent Unfortunately, in case of a $\odot$-similarity relation
Fix$(\underline{\theta})$ and Fix$(\overline{\theta})$ coincide only for 
a left-continuous t-norm $\odot$ and
the R-implicator $\vartriangleright$ induced by it.

\begin{theorem}
\label{thm:frs}
Assume that conditions in (ID) are satisfied
and let $(U,\theta)$ be a fuzzy approximation space with a $\odot$-similarity
relation $\theta$ of a finite range, and $F,G\in\mathcal{F}(U)$. Then $(F,G)$
is a fuzzy rough set induced by a fuzzy set with a finite range, if and only if
the following conditions hold:

\noindent (1) $G\in\ $Fix$(\underline{\theta})$, $F\in\ $Fix$(\overline{\theta})$,
$G\leq F$, and $F$ and $G$ have finite ranges;

\noindent (2) If $\mathcal{E}$ is a maximal $\varepsilon (G)$ class such that each
$\{a\}\subseteq\mathcal{E}$ is a maximal $E(F)$

class, then there exists an element $u\in\mathcal{E}$ such that $G(u)=F(u)$;

\noindent (3) If $E$ is a maximal $E(F)$ class such that each $\{a\}\subseteq E$ is a
maximal $\varepsilon (G)$

class, then there exists an element $v\in E$ such that $G(v)=F(v)$.
\end{theorem}

\begin{proof} By definition, $(F,G)$ is a fuzzy rough set if there
exists a map $f\in\mathcal{F}(U)$ such that $F=\overline{\theta}(f)$,
$G=\underline{\theta}(f)$. Suppose that $f$ has a finite range. We prove that
the conditions of Theorem \ref{thm:frs} are satisfied. Indeed,

(1) Property (ID) implies $\overline{\theta}(F)=F$, $\underline{\theta}(G)=G$, hence $F\in
\ $Fix$(\overline{\theta})$ and $G\in\ $Fix$(\underline{\theta})$. Clearly,
$G=\underline{\theta}(f)\leq\overline{\theta}(f)=F$. Because $f$ has a finite
range, in view of Lemma \ref{lem:frange}, $F$ and $G$ also have finite ranges.

\noindent In view of Corollary \ref{cor:fugufvgv}, conditions (2) and (3) are also satisfied.

Conversely, suppose that conditions (1), (2) and (3) are satisfied by $F$ and
$G$. In order to prove that $(F,G)$ is a fuzzy rough set, we will construct a
fuzzy set $f\in\mathcal{F}(U)$ with $F=\overline{\theta}(f)$, $G=\underline
{\theta}(f)$.

Since $F$ and $G$ are of a finite range, in view of Corollary \ref{cor:maxeqclexists}, for each
$E(F)$ class $E$ there exists a maximal $E(F)$ class
$E_{M}$ such that $E\leq_{R(F)}E_{M}$, and for any $\varepsilon (G)$
class $\mathcal{E}$ there is a maximal $\varepsilon (G)$ class
$\mathcal{E}_{M}$ with $\mathcal{E}\leq_{\varrho(G)}\mathcal{E}_{M}$. Denote the
family of maximal $\varepsilon (G)$ classes by $\{\mathcal{E}_{t}\mid
t\in T\}$. As a first step, from each class $\mathcal{E}_{t}$, $t\in T$ we
select exactly one element $a_{t}\in\mathcal{E}_{t}$ as follows:

\noindent 1) If $\mathcal{E}_{t}$ contains an element $p_{t}$ which does not belong to
any maximal $E(F)$ class, then we select it and set $a_{t}:=p_{t}$.

\noindent 2) If in $\mathcal{E}_{t}$ there is no element of type 1), however there
exists an element $q_{t}\in\mathcal{E}_{t}$ with $G(q_{t})=F(q_{t})$, then we
select it and set $a_{t}:=q_{t}$.

\noindent 3) If there are no elements of type 1) or 2) in $\mathcal{E}_{t}$, then we
select an element $r_{t}\in\mathcal{E}_{t}$ such that $\{r_{t}\}$ is not a
maximal $E(F)$ class, and we set $a_{t}:=r_{t}$.

\medskip

First, we show that we can always effectuate such a selection: assume
by contradiction that in some class $\mathcal{E}_{t}$ there are no elements
of type 1), 2) or 3). This means that for each $a_{t}\in\mathcal{E}_{t}$ the
set $\{a_{t}\}$ is a maximal $E(F)$ class. Then by Corollary \ref{cor:fugufvgv}(ii)
there exists an element $v\in\mathcal{E}_{t}$ with $F(v)=G(v)$. Since this
means that $v\in\mathcal{E}_{t}$ is of type 2), this is a contradiction.

As next step, we construct a fuzzy set $f\in\mathcal{F}(U)$ as follows:

\begin{equation}
    f(x)=\left\{
\begin{array}
[c]{l}%
G(x)\text{, if }x\in\{a_{t}\mid t\in T\}\text{;}\\
F(x)\text{, if }x\in U\setminus\{a_{t}\mid t\in T\}
\end{array}
\right. 
\tag{3}
\end{equation}

By its construction, $f$ also has a finite range. Now, we prove that in any
maximal $E(F)$ class $E$ there exists an element $u\in E$ with
$f(u)=F(u)$. By our construction, this would mean that any such class $E$
contains an element $x_{0}\in U\setminus\{a_{t}\mid t\in T\}$ or an element
$a_{t}=q_{t}\in E$ of type 2) with $f(q_{t})=G(q_{t})=F(q_{t})$.

By way of contradiction, assume that there is a maximal $E(F)$
class $E_{M}$ with $E_{M}\subseteq\{a_{t}\mid t\in T\}$ and $F(x)\neq
f(x)=G(x)$, for all $x\in E_{M}$. Then $E_{M}=\{a_{s}\mid s\in S\}$, for some
nonempty $S\subseteq T$. Observe that in this case $E_{M}$ cannot contain
elements of type 1) and 2). Hence, by our construction, for any element
$a_{s}$, $s\in S$ the set $\{a_{s}\}$ is not a maximal $E(F)$ class.
Thus $E_{M}$ is not a one-element set, i.e. $|S|\geq2$. Observe also,
that we can exclude the case when each element $a_{s},s\in S$ belongs to an
$\mathcal{E}_{t}$ class with a single element. Indeed, as in such case each
$\{a_{s}\}\subseteq E_{M}$ would be a maximal $\varepsilon (G)$ class,
and by Corollary \ref{cor:fugufvgv}(i) we would obtain $G(a_{s_{0}})=F(a_{s_{0}})$, for some
$a_{s_{0}}\in E_{M}$, contrary to our assumption. Hence there exists an
element $a_{s^{\ast}}\in E_{M}$ which was chosen from a maximal
$\varepsilon (G)$ class $\mathcal{E}_{s^{\ast}}$ with $|
\mathcal{E}_{s^{\ast}}|\geq2$. Since $a_{s^{\ast}}\in E_{M}\cap
\mathcal{E}_{s^{\ast}}$, in view of Corollary \ref{cor:assertions}(i), we have $E_{M}%
\subseteq\mathcal{E}_{s^{\ast}}$ or $\mathcal{E}_{s^{\ast}}\subseteq E_{M}$.
Since both $E_{M}$ and $\mathcal{E}_{s^{\ast}}$ have at least two elements,
both cases would imply that from the class $\mathcal{E}_{s^{\ast}}$ at least two elements had been
inserted into the set $\{a_{t}\mid t\in T\}$, in
contradiction to our construction for $\{a_{t}\mid t\in T\}$.

Thus we proved that in any maximal $E(F)$ class $E$ there is an
element $u\in E$ with $f(u)=F(u)$. It is also clear, that by our construction
from each maximal $\varepsilon (G)$ class $\mathcal{E}_{t}$, $t\in T$
an element $v=a_{t}\in\mathcal{E}_{t}$ had been selected with $f(v)=G(v)$.
Since by definition $G\leq f\leq F$, applying Proposition \ref{prop:hufuhvgv} we obtain
$F=\overline{\theta}(f)$, $G=\underline{\theta}(f)$, and our proof is
completed. \end{proof}

\section{Further properties of $E(F)$ and $\varepsilon
(G)$ classes}
\label{sec:propofeqc}

In this section we deduce some additional properties of $E(F)$ and $\varepsilon
(G)$ classes which will be used to prove our main Theorem \ref{thm:lattice}.  In the whole section we assume that for all $x,y\in\lbrack0,1]$ condition\
\begin{equation}
x\odot y=\text{min}(x,y)\text{, }x\vartriangleright y:=\ \text{max}%
(n(x),y)\text{, }n\text{ is an involutive negator} \tag{C}%
\end{equation}

\noindent holds, and that $\theta$ is a similarity relation. Then
$\vartriangleright$ is an S-implicator, and for $n(x)=1-x$, we re-obtain the
Kleene-Dienes implicator, therefore our $\vartriangleright$ is an extension of
it. Clearly, if (C) holds then (D) and (ID) are also satisfied. 

\begin{proposition}
\label{prop:maxiff}
(i) $E\subseteq U$ is a maximal
$E(F)$ class if and only if $\theta(a,z)<F(a)=F(b)\leq\theta(a,b)$ ,
for all $a,b\in E$ and $z\notin E$;

\noindent(ii) $\mathcal{E}\subseteq U$ is a maximal $\varepsilon (G)$
class, if and only if $n(\theta(a,b))\leq G(a)=G(b)<n(\theta(a,z))$, for all
$a,b\in\mathcal{E}$ and $z\notin\mathcal{E}$.
\end{proposition}

\begin{proof} (i) If $E\subseteq U$ is a maximal $E(F)$ class,
then by Lemma \ref{lem:assertions}(i) and (iii), we have $F(a)=F(b)\leq
\theta(a,b)$ and $\theta(a,z)<F(a)$, for all $a,b\in E$ and $z\notin E$.

Conversely, let $E\subseteq U$ and assume that all the relations from (i) are
satisfied. Then, $F(a)=$ min$(\theta(a,b),F(b))$, for all $a,b\in E$, i.e. we
get $(a,b)\in R(F)$ for all $a,b\in E$, and in view of Corollary \ref{cor:cor1}, we have
$(a,z)\notin R(F)$ for all $z\notin E$. Hence $(a,b)\in R(F)\cap
R(F)^{-1}=E(F)$ holds for all $a,b\in E$, and $(a,z)\notin E(F)$ for each
$z\notin E$. This means that $E$ is an $E(F)$ class. We also get
$E\nleq\lbrack z]_{E(F)}$ for all $z\notin E$, because $(a,z)\notin R(F)$.
Thus $E$ is a maximal $E(F)$ class.

\noindent(ii) If $\mathcal{E}\subseteq U$ is a maximal $\varepsilon
(G)$ class, then in view of Lemma \ref{lem:assertions}(iv) $n(\theta(a,b))\leq G(a)=G(b)$, for
all $a,b\in\mathcal{E}$ and we have $\mathcal{E}\nleq\lbrack
z]_{\varepsilon (G)}$, for any $z\notin\mathcal{E}$. Then $(a,z)\notin\varrho(G)$
implies $G(a)<n(\theta(a,z))$, according to Lemma \ref{lem:assertions}(iv). The converse
implication is proved analogously as in (i). \end{proof}

\begin{lemma}
\label{lem:outer}
Let $F=\overline{\theta}(f)$, $G=\underline
{\theta}(g)$, for some $f,g\in\mathcal{F}(U)$, and let $E$ be an
$E(F)$\ class and $\mathcal{E}$ be an $\varepsilon (G)$ class
such that $E\cap\mathcal{E}\neq\emptyset$. Then the following assertions hold:

\noindent(i) If $(x,y)\notin R(F)$, for some $x\in E$ and $y\in\mathcal{E}$,
then $(x,z)\in R(F)$ implies $(y,z)\in\varrho(G)$, for all $z\notin
E\cup\mathcal{E}$.

\noindent(ii) If $(x,y)\notin\varrho(G)$, for some $x\in\mathcal{E}$ and $y\in
E$, then $(x,z)\in\varrho(G)$ implies $(y,z)\in R(F)$, for all $z\notin
E\cup\mathcal{E}$.
\end{lemma}

\begin{proof} Let $a\in E\cap\mathcal{E}$. (i) Observe that the
relations $(x,y)\notin R(F)$ and $(x,a)\in R(F)$ exclude $(a,y)\in R(F)$. Thus
$(a,y)\notin R(F)$ yields $F(a)>\theta(a,y)$, by Corollary \ref{cor:cor1}.
Now, let $(x,z)\in R(F)$ and assume by contradiction $(y,z)\notin\varrho(G)$. Then
$a,y\in\mathcal{E}$ and $z\notin\mathcal{E}$ imply $\theta(a,z)<\theta(a,y)$.
On the other hand, $(a,x)\in R(F)$ and $(x,z)\in R(F)$ imply $(a,z)\in R(F)$.
Hence, Proposition \ref{prop:maxiff}(i) yields $F(a)\leq\theta(a,z)<\theta(a,y)$, a
contradiction to $F(a)>\theta(a,y)$. This proves $(y,z)\in\varrho(G)$.

\noindent(ii) By Proposition \ref{prop:prop1}(iv) we have $\varrho(G)=R(n(G))$,
$R(F)=\varrho(n(F))$, and by Lemma \ref{lem:assertions} (v), $\mathcal{E}$ is an $E%
(n(G))$ class, and $E$ is $\ $an $\varepsilon (n(F))$ class. Hence
$(x,y)\notin\varrho(G)$ for some $x\in\mathcal{E}$ and $y\in E$ and $(x,z)\in
\varrho(G)$ is equivalent to $(x,y)\notin R(n(G))$ and $(x,z)\in R(n(G))$,
therefore, $n(G)=n(\underline{\theta}(g))=\overline{\theta}(n(g))$ and
$n(F)=n(\overline{\theta}(f))=\underline{\theta}(n(f))$ form a pair that replaces
in the context of (ii) the pair $(F,G)$ from (i).
Thus $(y,z)\in\varrho(n(F))=R(F)$, in view of (i). \end{proof}

\begin{corollary}
\label{cor:outer}
Let $F=\overline{\theta}(f)$, $G=\underline
{\theta}(g)$, for some $f,g\in\mathcal{F}(U)$ and let $E$ be an $E%
(F)$\ class and $\mathcal{E}$ an $\varepsilon (G)$ class such that
$E\cap\mathcal{E}\neq\emptyset$.

\noindent(i) If $\mathcal{E}$ is a maximal $\varepsilon (G)$ class,
then $(x,y)\in R(F)$ for all $x\in E$ and $y\in\mathcal{E}$ or $E\subseteq
\mathcal{E}$ and there is no $t\in E$ and $z\notin\mathcal{E}$ with $(t,z)\in
R(F)$.

\noindent(ii) If $\mathcal{E}$ is a maximal $\varepsilon (G)$ class with $E\varsubsetneqq\mathcal{E}$ and there is no element $x\in E$
and $y\in\mathcal{E}\setminus E$ with $(x,y)\in R(F)$, then $E$ is a maximal
$E(F)$\ class.

\noindent(iii) If $E$ is a maximal $E(F)$ class, then $(x,y)\in\varrho(G)$ for
all $x\in\mathcal{E}$ and $y\in E$ or $\mathcal{E}\subseteq E$ and there is no
$t\in\mathcal{E}$ and $z\notin E$ with $(t,z)\in\varrho(G)$.

\noindent(iv) If  $E$ is a maximal $E(F)$ class such that $\mathcal{E}\varsubsetneqq E$ and there is no element
$x\in\mathcal{E}$ and $y\in E\setminus\mathcal{E}$ with $(x,y)\in\varrho(G)$,
then $\mathcal{E}$ is a maximal $\varepsilon (G)$\ class.
\end{corollary}

\begin{proof}(i) Let $\mathcal{E}$ be a maximal $\varepsilon (G)$
class and assume $E\nsubseteqq\mathcal{E}$. Then there exists $a\in
E\setminus\mathcal{E}$, and for all $y\in\mathcal{E}$, $(y,a)\notin\varrho(G)$ by
maximality of $\mathcal{E}$. Hence, in view of Corollary \ref{cor:assertions}(ii) we have
$(a,y)\in R(F)$, and because $(x,a)\in R(F)$ for each $x\in E$, we get
$(x,y)\in R(F)$ for all $x\in E$ and $y\in\mathcal{E}$. Consider now the case
when $E\subseteq\mathcal{E}$ and there are $x\in E$, $y\in\mathcal{E}$
with $(x,y)\notin R(F).$ Then $E\cup\mathcal{E=E}$. Assume that there exist
some elements $t\in E$ and $z\notin\mathcal{E}$ with $(t,z)\in R(F)$. Then
$(x,t)\in R(F)$ also yields $(x,z)\in R(F)$. Now, applying Lemma \ref{lem:outer}(i) we
obtain $(y,z)\in\varrho(G)$. Since $\mathcal{E}$ is a maximal
$\varepsilon (G)$ class and $z\notin\mathcal{E}$, this is not
possible, and this means that the second part of (i) holds.

\noindent(ii) Suppose that for all $x\in E$ and $y\in\mathcal{E\setminus}E$ we
have $(x,y)\notin R(F)$, and let $z\notin\mathcal{E}$. In view of Lemma
\ref{lem:outer}(i), $(x,z)\in R(F)$ for some $x\in E$ would imply $(y,z)\in\varrho(G)$, for
all $y\in\mathcal{E}$ - in contradiction to the fact that $\mathcal{E}$ is a
maximal $\varepsilon (G)$ class. Thus we deduce $(x,z)\notin R(F)$, for all
$x\in E$ and $z\notin E$. This means that $E$ is a maximal $E(F)$ class.

\noindent The proofs of (iii) and (iv) are duals of the proofs of (i) and (ii).
\end{proof}

\begin{proposition}
\label{prop:lowerupper}
Let $\theta$ be a similarity relation with
a finite range, and $F=\overline{\theta}(f)$, $G=\underline{\theta}(g)$, for
some $f,g\in\mathcal{F}(U).$

\noindent(i) If $\{a\}\subseteq U$ is a maximal $E(F)$\ class, then for any
$h\in\mathcal{F}(U)$ with $\overline{\theta}(h)(a)\geq F(a)$, we have
$\overline{\theta}(h)(a)=h(a)$.

\noindent(ii) If $\{b\}\subseteq U$ is a maximal $\varepsilon (G)$ class, then
for any $h\in\mathcal{F}(U)$ with $\underline{\theta}((h)(b)\leq G(b)$, we
have $\underline{\theta}(h)(b)=h(b)$.
\end{proposition}

\begin{proof} (i) If $\{a\}\subseteq U$ is a maximal $E%
(F)$ class, then $F(a)>\theta(a,y)$, for all $y\in U$, $y\neq a$, according to
Corollary \ref{cor:cor1}. Now, we can write:

$\overline{\theta}(h)(a)=\bigvee\{$min$(\theta(a,y),h(y))\mid y\in U\}=$

$h(a)\vee\left(  \bigvee\{\text{min}(\theta(a,y),h(y))\mid y\in U\setminus
\{a\}\}\right)  $,

\noindent and $\bigvee\{$min$(\theta(a,y),h(y))\mid y\in U\setminus\{a\}\}\leq
\bigvee\{\theta(a,y)\mid y\in U\setminus\{a\}\}<F(a)\leq\overline{\theta
}(h)(a)$, because $\theta$ is of a finite range. This implies $\overline
{\theta}(h)(a)=h(a)$.

\noindent(ii) If $\{b\}\subseteq U$ is a maximal $\varepsilon (G)$ class, then
$G(b)<n(\theta(b,y))$, for all $y\in U$, $y\neq b$, according to Corollary
\ref{cor:cor1}. We can write:

$\underline{\theta}((h)(b)=\bigwedge\{$max$(n(\theta(b,y)),h(y))\mid y\in U\}=$

$h(b)\wedge\left(  \bigwedge\{\text{max}(n(\theta(b,y)),h(y))\mid y\in
U\setminus\{b\}\}\right)  $,

\noindent and $\bigwedge\{$max$(n(\theta(b,y)),h(y))\mid y\in U\setminus\{b\}\}\geq
\bigwedge\{n(\theta(b,y))\mid y\in U\setminus\{b\}\}>G(b)\geq$ $\underline
{\theta}(h)(b)$, since $\theta$ is of a finite range. This yields
$\underline{\theta}(h)(b)=h(b)$.\end{proof}

\section{The lattice of fuzzy rough sets}
\label{sec:latticefrs}

Clearly, fuzzy rough sets corresponding to an approximation space $(U,\theta
)$, can be ordered as follows:%
\begin{equation}
\left(  \underline{\theta}(f),\overline{\theta}(f)\right)  \leq\left(
\underline{\theta}(g),\overline{\theta}(g)\right)  \Leftrightarrow
\underline{\theta}(f)\leq\underline{\theta}(g)\text{ and }\overline{\theta
}(f)\leq\overline{\theta}(g)\text{,}%
\tag{4}
\end{equation}

\noindent obtaining a poset $(\mathcal{FR}%
(U,\theta),\leq)$. If $\theta$ is reflexive, then $(\mathbf{0},\mathbf{0})$ and $(\mathbf{1},\mathbf{1})$ are its least and greatest elements.
If conditions in (D) hold, $n(\overline
{\theta}(f))=\underline{\theta}(n(f))$ and $n(\underline{\theta}%
(f))=\overline{\theta}(n(f))$ imply $(n(\overline{\theta}(f),n(\underline
{\theta}(f))\in\mathcal{FR}(U,\theta)$, for all $f\in\mathcal{F}(U)$. As $n$
is an involutive negator, $\Phi\colon\mathcal{FR}(U,\theta
)\rightarrow\mathcal{FR}(U,\theta)$, $\Phi(\left(  \underline{\theta
}(f),\overline{\theta}(f)\right)  )=(n(\overline{\theta}(f),n(\underline
{\theta}(f))$ is an involution, i.e. $\Phi(\Phi\left(  \underline{\theta
}(f),\overline{\theta}(f)\right)  )=\left(  \underline{\theta}(f),\overline
{\theta}(f)\right)  $. Since $\ \left(  \underline{\theta}(f),\overline
{\theta}(f)\right)  \leq\left(  \underline{\theta}(g),\overline{\theta
}(g)\right)  \Leftrightarrow(n(\overline{\theta}(g),n(\underline{\theta
}(g))\leq(n(\overline{\theta}(f),n(\underline{\theta}(f))$, we have
\[
\left(  \underline{\theta}(f),\overline{\theta}(f)\right)  \leq\left(
\underline{\theta}(g),\overline{\theta}(g)\right)  \Leftrightarrow\Phi\left(
\underline{\theta}(g),\overline{\theta}(g)\right)  \leq\Phi\left(
\underline{\theta}(f),\overline{\theta}(f)\right)  \text{,}%
\]
\noindent meaning that $\Phi$ is a dual order-isomorphism. Thus
$(\mathcal{FR}(U,\theta),\leq)$ is a self-dual poset, whenever conditions in (D) hold. 
In this section we will deduce some conditions under which $(\mathcal{FR}(U,\theta),\leq)$ is a lattice. 

\bigskip

Now let $L$ be a complete sublattice of $[0,1]$, and let $\mathcal{F}(U,L)$
stand for the family of all fuzzy sets $f\colon U\rightarrow L$.
The system of all $f\in\mathcal{F}(U,L)$ with a finite range is denoted
by $\mathcal{F}_{fr}(U,L)$. If $L=[0,1]$, then we write simply $\mathcal{F}%
_{fr}(U)$. As $0,1\in L$, we have $\mathbf{0,1}\in\mathcal{F}_{fr}(U,L)$. It
is obvious that for any $f_{1},f_{2}\in\mathcal{F}_{fr}(U,L)$, $f_{1}\vee
f_{2}=$ max$\left(  f_{1},f_{2}\right)  $ and $f_{1}\wedge f_{2}=$
min$\left(  f_{1},f_{2}\right)  $ are of a finite range and their values are in
$L$, hence $(\mathcal{F}_{fr}(U,L),\leq)$ is a bounded distributive lattice.
Clearly, for any $f\in\mathcal{F}(U,L)$ with a finite range and any negator $n$, the
fuzzy set $n(f)$ also has a finite range, i.e. $n(f) \in \mathcal{F}_{fr}(U,L)$. Further, if relation $\theta$ has a
finite range, then in view of Lemma \ref{lem:frange}, for any $f\in\mathcal{F}_{fr}(U)$:
$\underline{\theta}(f),\overline{\theta}(f)\in\mathcal{F}_{fr}(U)$. \ In all what follows, suppose that condition (C) holds with $n(L) \subseteq L $, and $\theta\colon U\times U\rightarrow L$ is a similarity
relation. Then

\begin{center}
$\overline{\theta}(f)(x)=\bigvee\{$min$(\theta(x,y),f(y))\mid y\in U\}$ and
$\underline{\theta}(f)(x)=\bigwedge\{$max$(n(\theta(x,y)),f(y))\mid y\in U\}$,
\end{center}

\noindent\ for all $x\in U$. As $L$ is closed w.r.t. arbitrary joins and
meets, and $n(L) \subseteq L $, we get that $ \underline{\theta}(f),\overline{\theta
}(f)  \in\mathcal{F}_{fr}(U,L)$. Now consider the poset  defined on 
\[
\mathcal{H}:=\{\left(  \underline{\theta}(f),\overline{\theta}(f)\right)  \mid
f\in\mathcal{F}_{fr}(U,L)\}\text{.}%
\]

\noindent We will prove that $(\mathcal{H},\leq)$ is
a lattice, moreover if $U$ or $L$ is finite, then it is a complete lattice. This approach
is motivated by the following examples:

1) If $U$ is a finite set, then $\theta$ and all $f\in\mathcal{F}(U)$
have finite ranges. Hence for $L=[0,1]$ we have $\mathcal{F}_{fr}%
(U,L)=\mathcal{F}(U)$, and $(\mathcal{H},\leq)$ equals
to $(\mathcal{FR}(U,\theta),\leq)$.

2) If $L$ is a finite chain with $0,1\in L$, then any $f\in\mathcal{F}(U,L)$
has a finite range, hence $\mathcal{F}_{fr}(U,L)=\mathcal{F}(U,L)$, and
$(\mathcal{H},\leq)$ is the same as $(\mathcal{F}\mathcal{R}(U,L),\leq)$.

\begin{remark}
\label{rem:complete}
(a) The relations $\underline{\theta}(f_{1}%
)\wedge\underline{\theta}(f_{1})=$ $\underline{\theta}\left(  f_{1}\wedge
f_{2}\right)  $ and $\overline{\theta}(f_{1})\vee\overline{\theta}%
(f_{2})=\overline{\theta}(f_{1}\vee f_{2})$ always hold (see e.g. \cite{DVCG}) for
any $f_{1},f_{2}\in\mathcal{F}(U)$. Assume now that condition (C) holds, or $\odot$ is a left continuous t-norm and $\vartriangleright$ is its $R$-implicator. It is known (see e.g. \cite{MH}) that in this case the equalities

$\bigwedge\{$ $\underline{\theta}(f_{i})\mid i\in I\}=\ \underline{\theta
}\left(  \bigwedge\{f_{i}\mid i\in I\}\right)  $, $\bigvee\{\overline{\theta
}(f_{i})\mid i\in I\}=\overline{\theta}\left(  \bigvee\{f_{i}\mid i\in
I\}\right)  $

\noindent also hold for any (nonempty) system $f_{i}\in\mathcal{F}(U)$, $i\in
I$.

\noindent (b) If now $L\subseteq\lbrack0,1]$ is a complete lattice and
$\theta\colon U\times U\rightarrow L$, then clearly, for any $f_{i}%
\in\mathcal{F}(U,L)$, $i\in I$ we get $\bigwedge\{f_{i}\mid i\in
I\},\bigvee\{f_{i}\mid i\in I\}\in\mathcal{F}(U,L)$ and $\bigwedge\{$ $\underline{\theta}(f_{i})\mid i\in I\}=\underline
{\theta}\left(  \bigwedge\{f_{i}\mid i\in I\}\right) \in\mathcal{F}(U,L) $, $\overline{\theta
}\left(  \bigvee\{f_{i}\mid i\in I\}\right)  \in\mathcal{F}(U,L)$. \\
\noindent (c) As in this case conditions from (ID) also
hold, in view of [4], for a $\odot$-similarity relation
$\theta$, $f\mapsto\underline{\theta}(f)$, $f\in\mathcal{F}(U,L)$ is an interior operator, and the map
$f\mapsto\overline{\theta}(f)$, $f\in\mathcal{F}(U,L)$ is a closure
operator. Hence $\left(  \text{Fix}_{L}\left(  \underline{\theta}\right)
,\leq\right)  $ and $\left(  \text{Fix}_{L}\left(  \overline{\theta}\right)
,\leq\right)  $ are complete lattices, where Fix$_{L}\left(  \underline{\theta
}\right)  :=\{f\in\mathcal{F}(U,L)\mid\underline{\theta}(f)=f\}$ and
Fix$_{L}\left(  \overline{\theta}\right)  :=\{f\in\mathcal{F}(U,L)\mid
\overline{\theta}(f)=f\}$.
\end{remark}

\begin{proposition}
\label{prop:completesubl}
Assume that conditions in (ID) are
satisfied, and let $L\subseteq\lbrack0,1]$ be a complete lattice and
$\theta\colon U\times U\rightarrow L$ be a $\odot$-similarity relation. Then
$\left(  \text{Fix}_{L}\left(  \overline{\theta}\right)  ,\leq\right)  $ and
$\left(  \text{Fix}_{L}\left(  \underline{\theta}\right)  ,\leq\right)  $ are
complete sublattices of $\mathcal{F}(U,L)$.
\end{proposition}

\begin{proof} Let $f_{i}\in\ $Fix$_{L}\left(  \overline{\theta
}\right)  $, $i\in I$ arbitrary. Then, in view of Remark \ref{rem:complete}, $\bigvee
\{f_{i}\mid i\in I\}\in\mathcal{F}(U,L)$, and $\overline{\theta}\left(
\bigvee\{f_{i}\mid i\in I\}\right)  =\bigvee\{\overline{\theta}(f_{i})\mid
i\in I\}=\bigvee\{f_{i}\mid i\in I\}$. Hence $\bigvee\{f_{i}\mid i\in I\}\in
$\ Fix$_{L}\left(  \overline{\theta}\right)  $. As Fix$_{L}\left(  \overline{\theta}\right) $ is the system of closed sets of the operator $ f\mapsto\overline{\theta
}(f)$ and $f_{i} \in\ $Fix$_{L}\left(  \overline{\theta
}\right) $, $i\in I$,
we also have $\bigwedge\limits_{i\in I}f_{i}\in\ $%
Fix$_{L}\left(  \overline{\theta}\right)  $. Hence $\left(  \text{Fix}%
_{L}\left(  \overline{\theta}\right)  ,\leq\right)  $ is a complete sublattice
of $\left(  \mathcal{F}(U,L),\leq\right)  $. The claim that $\left(
\text{Fix}_{L}\left(  \underline{\theta}\right)  ,\leq\right)  $ is complete sublattice
of $\left(  \mathcal{F}(U,L),\leq\right)  $ is proved dually.\end{proof}

\begin{corollary}
\label{cor:fafia}
Let $\theta\colon U\times U\rightarrow L$ be
a similarity relation with a finite range on $U$, $f_{i}\in\mathcal{F}(U,L)$,
$i\in I$, $F=%
{\textstyle\bigwedge}
\{\overline{\theta}(f_{i})\mid i\in I\}$ and let $\{a\}\subseteq U$ be a
maximal $E(F)$ class. Then $F(a)=\bigwedge\{f_{i}(a)\mid i\in I\}$.
\end{corollary}

\begin{proof} In view of Proposition \ref{prop:completesubl} we have $F=%
{\textstyle\bigwedge}
\{\overline{\theta}(f_{i})\mid i\in I\}\in$\ Fix$_{L}(\overline{\theta})$,
i.e. $F=\overline{\theta}(F)$. Since $\overline{\theta}(f_{i})(a)\geq F(a)$,
$i\in I$, by using Proposition \ref{prop:lowerupper}(i) we obtain $\overline{\theta}%
(f_{i})(a)=f_{i}(a)$, for all $i\in I$. This yields $F(a)=\bigwedge
\{f_{i}(a)\mid i\in I\}$.  \end{proof}

\begin{theorem}
\label{thm:lattice}
Let $\theta\colon U\times U\rightarrow L $ be a similarity relation
of a finite range, and assume that condition (C) holds with a negator satisfying $ n(L)\subseteq L$. 

\noindent(i) If the fuzzy sets $\bigwedge\limits_{i\in I}f_{i}$,
$\underset{i\in I}{\bigwedge}\overline{\theta}(f_{i})$, $f_{i}\in\mathcal{F}(U,L)$, $i\in
I$ have finite ranges, then the infimum of fuzzy rough sets $\left(
\underline{\theta}(f_{i}),\overline{\theta}(f_{i})\right)  $, $i\in I$ exists
in $(\mathcal{FR}(U,L),\leq)$ and its components have finite ranges.

\noindent(ii)  $(\mathcal{H},\leq)=(\{\left(  \underline{\theta}(f),\overline{\theta}(f)\right)
\mid f\in\mathcal{F}_{fr}(U,L)\},\leq)$ is a lattice.

\noindent(iii) If $U$ or $L$ is finite, then $(\mathcal{FR}(U,L),\leq)$ is a
complete lattice.
\end{theorem}

\begin{proof}(i) Denote $G=\underline{\theta}(\bigwedge
\limits_{i\in I}f_{i})$ and $F=\underset{i\in I}{\bigwedge}\overline{\theta
}(f_{i})$. Then $G,F \in \mathcal{F}(U,L)$, by Remark \ref{rem:complete}(b), and we have $\underline{\theta}(G)=G$ and $\overline{\theta}(\overline{\theta
}(f_{i}))=\overline{\theta}(f_{i})$, $ i\in I $, according to Remark \ref{rem:complete}(c). Thus $G\in$\ Fix$_{L}(\underline{\theta
})$. Since 
$\theta$ and $\bigwedge\limits_{i\in I}f_{i}$ have finite ranges, $G$ also has
a finite range. As $\overline{\theta}(f_{i}) \in$  Fix$_{L}(\overline{\theta
})$, Proposition \ref{prop:completesubl} gives $F\in$\ Fix$_{L}(\overline{\theta
})$, and by assumption $F$ has a finite range. Clearly, $G=\underline{\theta
}(\bigwedge\limits_{i\in I}f_{i})\leq\overline{\theta}(f_{i})$, for all $i\in
I$, whence $G\leq F$. 
Using $G$ and $F$ we will construct a fuzzy set $f\in\mathcal{F}(U,L)$ such
that $\left(  \underline{\theta}(f),\overline{\theta}(f)\right)  $ equals to
inf$\{(\underline{\theta}(f_{i}),\overline{\theta}(f_{i}))\mid i\in I\}$.

\medskip
First, from each maximal $\varepsilon (G)$ class $\mathcal{E}_{t}$,
$t\in T$ we select exactly one element $b_{t}\in\mathcal{E}_{t}$ as follows:

\noindent 1) If $\mathcal{E}_{t}$ contains an element $q_{t}\in\mathcal{E}_{t}$ with
$G(q_{t})=F(q_{t})$, then we set $b_{t}:=q_{t}$.

\noindent 2) If there are no such elements in $\mathcal{E}_{t}$, however there exists
an $s_{t}\in\mathcal{E}_{t}$ such that $\{s_{t}\}$ is not an $E(F)$ class,
then we choose it and set $b_{t}:=s_{t}$.

\noindent 3) If there are no elements of type 1) or 2) in in $\mathcal{E}_{t}$, then
we select an element $r_{t}\in\mathcal{E}_{t}$ such that $\{r_{t}\}$ is not a
maximal $E(F)$ class, and we set $b_{t}:=r_{t}$.

\medskip

Now we show that we can always manage such a selection. Indeed, assume by
contradiction that in some class $\mathcal{E}_{z}$ there are no elements of
type 1), 2) and 3). This means that for each $x\in\mathcal{E}_{z}$ the set
$\{x\}$ is a maximal $E(F)$ class. Then in view of Corollary \ref{cor:fafia}, we
have $F(x)=\bigwedge\limits_{i\in I}f_{i}(x)$, for each $x\in\mathcal{E}_{z}$.
As $\bigwedge\limits_{i\in I}f_{i}(x)$ has a finite range, by Proposition \ref{prop:famaxfy}(ii) we get $G(y)=\ $min$\{\bigwedge\limits_{i\in I}%
f_{i}(x)\mid x\in\mathcal{E}_{z}\}$, for all $y\in\mathcal{E}_{z}$, because
$G=\underline{\theta}(\bigwedge\limits_{i\in I}f_{i})$. Hence, there exists an element $v\in\mathcal{E}_{z}$ such that
$G(v)=\bigwedge\limits_{i\in I}f_{i}(v)=F(v)$. Since this result means that
$v$ is an element of type 1) in $\mathcal{E}_{z}$, this is a contradiction.

As next step, we construct a fuzzy set $f\in\mathcal{F}(U,L)$ as
follows:

\begin{equation}
f(x)=\left\{
\begin{array}
[c]{l}%
G(x)\text{, if }x\in\{b_{t}\mid t\in T\}\text{;}\\
F(x)\text{, if }x\in U\setminus\{b_{t}\mid t\in T\}
\end{array}
\right.
\tag{5}
\end{equation}

As $G,F\in\mathcal{F}(U,L)$, we have $f\in\mathcal{F}(U,L)$. Since $F$ and $G$
have finite ranges, $f$ also has a finite range. As from each maximal
$\varepsilon (G)$ class $\mathcal{E}_{t}$, $t\in T$ an element
$b_{t}\in\mathcal{E}_{t}$ was selected and $f(b_{t})=G(b_{t})$, $f\geq G$
hold, by Proposition \ref{prop:hufuhvgv}(ii) we have $\underline{\theta}(f)=G=\underline
{\theta}(\bigwedge\limits_{i\in I}f_{i})$. We prove that $\left(
\underline{\theta}(f),\overline{\theta}(f)\right)  $ is the infimum of the
system $(\underline{\theta}(f_{i}),\overline{\theta}(f_{i})),i\in
I$. Thus we are going to show that $\left(  \underline{\theta}(f),\overline
{\theta}(f)\right)  $ is a lower bound of $(\underline{\theta}(f_{i}%
),\overline{\theta}(f_{i})),i\in I$ and for any $h\in\mathcal{F}(U,L)$ with
$(\underline{\theta}(h),\overline{\theta}(h))\leq(\underline{\theta}%
(f_{i}),\overline{\theta}(f_{i}))$, $i\in I$ we have $(\underline{\theta
}(h),\overline{\theta}(h))\leq\left(  \underline{\theta}(f),\overline{\theta
}(f)\right)  $. As by definition $f\leq F$, we also have $\overline{\theta
}(f)\leq\overline{\theta}(F)=F \leq\overline{\theta}(f_{i})$, $i\in I$. 
Since $\underline{\theta}(f)=\underline{\theta}(\bigwedge\limits_{i\in
I}f_{i})\leq\underline{\theta}(f_{i})$, $i\in I$, now $\left(  \underline
{\theta}(f),\overline{\theta}(f)\right)  $ is a lower bound of $(\underline
{\theta}(f_{i}),\overline{\theta}(f_{i})),i\in I$ and condition $\overline
{\theta}(h)\leq\overline{\theta}(f_{i})$, $i\in I$ is equivalent to
$\overline{\theta}(h)\leq\ \underset{i\in I}{%
{\textstyle\bigwedge}
}\overline{\theta}(f_{i})=F$. Since $\underline{\theta}(f)=\underline{\theta
}(\bigwedge\limits_{i\in I}f_{i})=\bigwedge\limits_{i\in I}\underline{\theta
}(f_{i})$, we also have

$\underline{\theta}(h)\leq\underline{\theta}(f_{i})$,
$i\in I \Longleftrightarrow  \underline{\theta}(h)\leq\bigwedge
\limits_{i\in I}\underline{\theta}(f_{i})=\underline{\theta}(f)=G$. 

\noindent Hence to
prove $(\underline{\theta}(h),\overline{\theta}(h))\leq\left(  \underline
{\theta}(f),\overline{\theta}(f)\right)  $, for all sets $h\in\mathcal{F}(U,L)$ with \\
$(\underline{\theta}(h),\overline{\theta}(h))\leq(\underline{\theta}%
(f_{i}),\overline{\theta}(f_{i}))$, $i\in I$, it is enough to show that
$\overline{\theta}(h)\leq$ $\overline{\theta}(f)$ holds for any $h\in
\mathcal{F}(U,L)$ with $\underline{\theta}(h)\leq G$ and $\overline{\theta
}(h)\leq F$.\medskip

Take any $h$ with this property and any $x\in U$. If $x\in U\setminus
\{b_{t}\mid t\in T\}$ or $x=b_{t_{0}}$ for some $t_{0}\in$ $T$ with
$G(b_{t_{0}})=F(b_{t_{0}})$, then $f(x)=F(x)$, hence $h(x)\leq\overline
{\theta}(h)(x)\leq F(x)=f(x)\leq$ $\overline{\theta}(f)(x)$.

Let $x=b_{t_{0}}$, for some $t_{0}\in T$ such that $G(b_{t_{0}})\neq
F(b_{t_{0}})$. Then $f(b_{t_{0}})=G(b_{t_{0}})$, by our construction. If
$\{b_{t_{0}}\}$ is a maximal $\varepsilon (G)$ class, then in view of
Proposition \ref{prop:lowerupper}(ii), $\underline{\theta}(h)(b_{t_{0}})\leq G(b_{t_{0}})$
implies $h(b_{t_{0}})=\underline{\theta}(h)(b_{t_{0}})\leq G(b_{t_{0}})$, i.e.
we obtain $h(x)\leq G(x)=f(x)\leq\overline{\theta}(f)(x)$.

Assume now that $\mathcal{E}_{t_{0}}$, the maximal $\varepsilon (G)$
class containing $b_{t_{0}}$, has at least two elements. Denote the $E(F)$ class containing $b_{t_{0}}$ by
$E_{0}$. 

If $E_{0}\nsubseteqq\mathcal{E}_{t_{0}}$, then there exists a $z_0\in
E_{0}\setminus\mathcal{E}_{t_{0}}$, and we have $(y,z_0)\notin\varrho(G)$ for each
$y\in\mathcal{E}_{t_{0}}$, because $\mathcal{E}_{t_{0}}$ is a maximal
$\varepsilon (G)$ class. Hence, by Corollary \ref{cor:assertions}(ii), we get $(z,y)\in R(F)$ for
all $z\in E_{0}$ and $y\in\mathcal{E}_{t_{0}}$. Thus $(b_{t_{0}},c)\in R(F)$ for any $c\in\mathcal{E}_{t_{0}}$,
$c\neq b_{t_{0}}$. Clearly, $c\notin\{b_{t}\mid t\in T\}$, because only a single
element $b_{t_{0}}$ was selected from $\mathcal{E}_{t_{0}}$, and hence
$f(c)=F(c)$. Since $\left(  b_{t_{0}},c\right)  \in R(F)$ and $f\leq F$, by
applying Proposition \ref{prop:prop2}(i) we get $\overline{\theta}(f)(b_{t_{0}})=F(b_{t_{0}})$.
Thus we obtain $h(b_{t_{0}})\leq\overline{\theta}(h)(b_{t_{0}})\leq
F(b_{t_{0}})=\overline{\theta}(f)(b_{t_{0}})$, i.e. $h(x)\leq\overline{\theta
}(f)(x)$.

If $E_{0}\subseteq\mathcal{E}_{t_{0}}$, then we claim that $(b_{t_{0}},e)\in
R(F)$ for some element $e\in\mathcal{E}_{t_{0}}\setminus\{b_{t_{0}}\}$ (such
an element exists, because $|\mathcal{E}_{t_{0}}|\geq2$).
Clearly, if $E_{0}$ has at least two elements, then $e$ can be chosen as any
element from $E_{0}\setminus\{b_{t_{0}}\}$. If $E_{0}=\{b_{t_{0}}\}$, then in
view of our construction, the element $b_{t_{0}}$ is of type 3), i.e.
$\{b_{t_{0}}\}$ is an $E(F)$ class which is not maximal. However, if
$(b_{t_{0}},e)\notin R(F)$ would hold for each $e\in\mathcal{E}_{t_{0}%
}\setminus\{b_{t_{0}}\}$, then in view of Corollary \ref{cor:outer}(ii), $\{b_{t_{0}}\}$
would be a maximal $E(F)$ class, contrary to our hypothesis. As no element
different from $b_{t_{0}}$ was selected from $\mathcal{E}_{t_{0}}$, we have
$e\notin\{b_{t}\mid t\in T\}$, and hence $f(e)=F(e)$. Since $(b_{t_{0}},e)\in
R(F)$, repeating now the
previous argument, we obtain again $h(b_{t_{0}})\leq$ $\overline{\theta
}(f)(b_{t_{0}})$, i.e. $h(x)\leq\overline{\theta}(f)(x)$.

Hence for each $x\in U$ we obtained $h(x)\leq\overline{\theta}(f)(x)$. Thus
$h\leq\overline{\theta}(f)$. In view of \cite{DVCG}, this implies
$\overline{\theta}(h)\leq\overline{\theta}\left(  \overline{\theta}(f)\right)
=\overline{\theta}(f)$. Thus $\left(  \underline{\theta}(f),\overline{\theta
}(f)\right)  $ is the infimum of $(\underline{\theta}(f_{i}),\overline{\theta
}(f_{i}))$, $i\in I$. Since $f\in\mathcal{F}(U,L)$ has a finite range,
$\underline{\theta}(f),\overline{\theta}(f)\in\mathcal{F}(U,L)$ also have
finite ranges.

(ii) For any $f_{1},f_{2}\in\mathcal{F}_{fr}(U,L)$, $f_{1}\wedge
f_{2}$ has a finite range. By Lemma \ref{lem:frange}, as $\overline{\theta}(f_{1}%
),\overline{\theta}(f_{2})\in\mathcal{F}_{fr}(U,L)$, $\overline{\theta}%
(f_{1})\wedge\overline{\theta}(f_{2})$ also has a finite range. Applying now
(i) with $I=\{1,2\}$, we get that $(\mathcal{H},\leq)$ is a
$\wedge$-semilattice. Since condition (C) implies property (D), $(\mathcal{H},\leq)$ is self-dual, and hence it is a lattice.

(iii) If $U$ or $L$ is finite, then $ \theta$ and each $f\in\mathcal{F}(U,L)$ have
finite ranges, i.e. $\mathcal{F}(U,L)=$ $\mathcal{F}_{fr}(U,L)$. As for any
$f_{i}$ $\in\mathcal{F}(U,L)$, $i\in I$ we have $\bigwedge\limits_{i\in
I}f_{i}$, $\underset{i\in I}{%
{\textstyle\bigwedge}
}\overline{\theta}(f_{i})\in\mathcal{F}(U,L)$, the fuzzy sets $\bigwedge
\limits_{i\in I}f_{i}$ and $\underset{i\in I}{%
{\textstyle\bigwedge}
}\overline{\theta}(f_{i})$ also have finite ranges. Hence, in view of (i),
inf$\{(\underline{\theta}(f_{i}),\overline{\theta}(f_{i}))\mid
i\in I\}$ always exists, i.e. $(\mathcal{H}$,$\leq)$ is a complete $\wedge
$-semilattice. Since $(\mathcal{H},\leq)$ is self-dual, it is a complete
lattice. \end{proof}

\begin{remark}
\label{rem:infsup}
If for a system $f_{i}\in\mathcal{F}(U,L)$,
$i\in I$ we have
$\left(  \underline{\theta}(f),\overline{\theta}(f)\right) $ = $\left(  \bigwedge\{\underline{\theta}\left(  f_{i}\right)
\mid i\in I\},\bigwedge\{\overline{\theta}\left(  f_{i}\right)  \mid i\in
I\}\right) $, for
an $f\in\mathcal{F}(U,L)$, then $\left(  \underline{\theta}(f),\overline
{\theta}(f)\right)  $ equals to the infimum of $  \left(
\underline{\theta}(f_{i}),\overline{\theta}(f_{i})\right), i\in
I  $. Indeed, for any $h\in\mathcal{F}(U,L)$ with $(\underline
{\theta}(h),\overline{\theta}(h))\leq(\underline{\theta}(f_{i}),\overline
{\theta}(f_{i}))$, $i\in I$ we get $(\underline{\theta}(h),\overline{\theta}(h))\leq\left(
\underline{\theta}(f),\overline{\theta}(f)\right) $, meaning that $\left(
\underline{\theta}(f),\overline{\theta}(f)\right)  $ is the infimum of
$  \left(
\underline{\theta}(f_{i}),\overline{\theta}(f_{i})\right), i\in
I  $. Analogously, $\left(
{\textstyle\bigvee}
\{\underline{\theta}\left(  f_{i}\right)  \mid i\in I\},%
{\textstyle\bigvee}
\{\overline{\theta}\left(  f_{i}\right)  \mid i\in I\}\right) $ is the
supremum of 
$  \left(  \underline{\theta}(f_{i}),\overline{\theta
}(f_{i})\right), i\in I$ whenever $\left(
{\textstyle\bigvee}
\{\underline{\theta}\left(  f_{i}\right)  \mid i\in I\},%
{\textstyle\bigvee}
\{\overline{\theta}\left(  f_{i}\right)  \mid i\in I\}\right) \in \mathcal{FR}(U,L)$. 
\end{remark}

\begin{example}
\label{ex:notfrs}
Here we show how a meet $\left(  \underline{\theta}(f_{1}),\overline{\theta}(f_{1})\right)
\wedge\left(  \underline{\theta}(f_{2}),\overline{\theta}(f_{2})\right)  $ can be calculated by using construction (5) in the proof of Theorem \ref{thm:lattice}. The similarity relation $\theta $ is given on Figure \ref{fig:theta4}, and $L=\{0, 0.1, 0.25, 0.5, 0.75, 1 \} $. The fuzzy sets $f_{1},f_{2}$ and their approximations are given in Table \ref{tab:fuzzy3}.

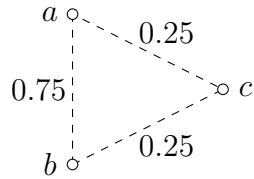
\begin{figure}[H]
    \centering
    \begin{tikzpicture}

    \draw[dashed] (0, -1) -- (0, 1);
    \draw[dashed] (0, -1) -- (2, 0) -- (0, 1);

    \draw (-0.3, 1) node {$a$};
    \draw (-0.3, -1) node {$b$};
    \draw (2.3, 0) node {$c$};

    \draw (-0.45, 0) node {$0.75$};
    \draw (1.25, 0.75) node {$0.25$};
    \draw (1.25, -0.75) node {$0.25$};
    
    \draw[fill=white] (0, -1) circle [radius=2pt];
    \draw[fill=white] (0, 1) circle [radius=2pt];
    \draw[fill=white] (2, 0) circle [radius=2pt];
    
    \end{tikzpicture}
    \caption{The fuzzy similarity relation $\theta$ of Example \ref{ex:notfrs}}
    \label{fig:theta4}
\end{figure}

 \begin{table}[H]
    \centering
    \begin{tabular}{|P{1.5cm}|P{0.75cm}|P{0.75cm}|P{0.75cm}|}
        \hline
         $u$ & $a$ & $b$ & $c$  \\ \hline
         $f_1(u)$ & 1 & 0.1 & 0.5  \\ \hline
         $\overline{\theta}(f_1)(u)$ & 1 & 0.75 & 0.5  \\ \hline
         $\underline{\theta}(f_1)(u)$ & 0.25 & 0.1 & 0.5 \\ \hline
    \end{tabular}
    \hspace{0.75cm}
    \begin{tabular}{|P{1.5cm}|P{0.75cm}|P{0.75cm}|P{0.75cm}|}
        \hline
         $u$ & $a$ & $b$ & $c$  \\ \hline
         $f_2(u)$ & 0.1 & 1 & 0.5  \\ \hline
         $\overline{\theta}(f_2)(u)$ & 0.75 & 1 & 0.5  \\ \hline
         $\underline{\theta}(f_2)(u)$ & 0.1 & 0.25 & 0.5 \\ \hline
    \end{tabular}
    \caption{The fuzzy sets $f_1$ and $f_2$ of Example \ref{ex:notfrs} and their approximations}
    \label{tab:fuzzy3}
\end{table}

The corresponding fuzzy rough sets are represented in the form  $\alpha_1=
\begin{pmatrix}
1 & 0.75 & 0.5\\
0.25 & 0.1 & 0.5
\end{pmatrix}
$ and $\alpha_2=
\begin{pmatrix}
0.75 & 1 & 0.5\\
0.1 & 0.25 & 0.5
\end{pmatrix}
$, where the first row stands for the upper approximations and the second row shows their lower approximations. Computing the meets $F = \overline{\theta}(f_1) \wedge \overline{\theta}(f_2)$ and $G = \underline{\theta}(f_1) \wedge \underline{\theta}(f_2)$, we obtain the pair $\begin{pmatrix}
F \\
G
\end{pmatrix}=
\begin{pmatrix}
0.75 & 0.75 & 0.5\\
0.1 & 0.1 & 0.5
\end{pmatrix}
$, which is not a fuzzy rough set. The quasiorders induced by $F$ and $G$ are given in Figure \ref{fig:rhor}.

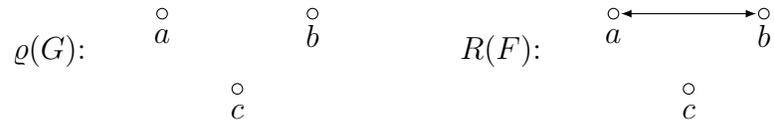
\begin{figure}[H]
    \centering
    \begin{tikzpicture}

    \draw[fill=white] (0, 0) circle [radius=2pt];
    \draw[fill=white] (1, 1) circle [radius=2pt];
    \draw[fill=white] (-1, 1) circle [radius=2pt];

    \draw (0, -0.3) node {$c$};
    \draw (1, 0.7) node {$b$};
    \draw (-1, 0.7) node {$a$};

    \draw (-2.5, 0.5) node {$\varrho(G)$:};

    \draw[latex-latex] (5.1, 1) -- (6.9, 1);

    \draw[fill=white] (6, 0) circle [radius=2pt];
    \draw[fill=white] (7, 1) circle [radius=2pt];
    \draw[fill=white] (5, 1) circle [radius=2pt];

    \draw (6, -0.3) node {$c$};
    \draw (7, 0.7) node {$b$};
    \draw (5, 0.7) node {$a$};
    
    \draw (3.5, 0.5) node {$R(F)$:};
    
    \end{tikzpicture}
    
    \caption{The quasiorders $\varrho(G)$ and $R(F)$.}
    \label{fig:rhor}
\end{figure}

Observe that each element is a maximal $\varepsilon (G)$ class. Hence, applying formula (5) from the proof of Theorem \ref{thm:lattice}, we obtain the reference set $ f:=G$, and as a corresponding fuzzy rough set $\begin{pmatrix}
0.25 & 0.25 & 0.5\\
0.1 & 0.1 & 0.5
\end{pmatrix}
$.
\end{example}

\begin{example}
\label{ex:distlattice}
Let us consider the similarity relation $\theta $ on Figure \ref{fig:theta3}, and set $L= \{0, 0.5, 1 \}$. The lattice $(\mathcal{H},\leq)$ of fuzzy rough sets is shown on Figure \ref{fig:lattice}.

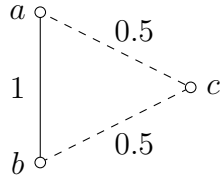
\begin{figure}[H]
    \centering
    \begin{tikzpicture}

    \draw (0, -1) -- (0, 1);
    \draw[dashed] (0, -1) -- (2, 0) -- (0, 1);

    \draw (-0.3, 1) node {$a$};
    \draw (-0.3, -1) node {$b$};
    \draw (2.3, 0) node {$c$};

    \draw (-0.3, 0) node {$1$};
    \draw (1.25, 0.75) node {$0.5$};
    \draw (1.25, -0.75) node {$0.5$};
    
    \draw[fill=white] (0, -1) circle [radius=2pt];
    \draw[fill=white] (0, 1) circle [radius=2pt];
    \draw[fill=white] (2, 0) circle [radius=2pt];
    
    \end{tikzpicture}
    \caption{The similarity relation $\theta$ of Example \ref{ex:distlattice}}
    \label{fig:theta3}
\end{figure}

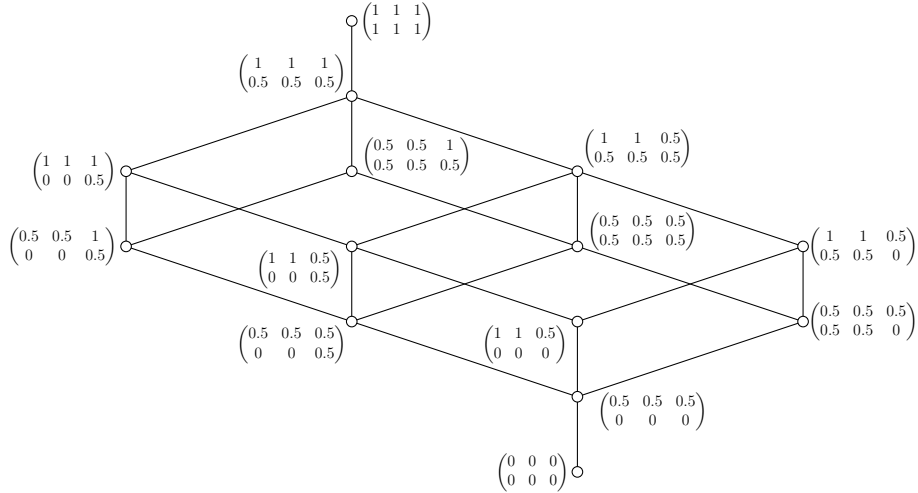
\begin{figure}[H]
    \centering
    \begin{tikzpicture}

    \draw (0, 0) -- (0, 1) -- (3, 2) -- (3, 3) -- (0, 4) -- (-3, 5) -- (-3, 6);
    \draw (0, 1) -- (-3, 2) -- (-6, 3) -- (-6, 4) -- (-3, 5);
    \draw (0, 1) -- (0, 2) -- (-3, 3) -- (-6, 4);
    \draw (3, 2) -- (0, 3) -- (-3, 4);
    \draw (0, 2) -- (3, 3);
    \draw (-3, 2) -- (0, 3);
    \draw (-3, 2) -- (-3, 3);
    \draw (-3, 3) -- (0, 4);
    \draw (0, 3) -- (0, 4);
    \draw (-6, 3) -- (-3, 4) -- (-3, 5);
    
    \draw[fill=white] (0, 0) circle [radius=2pt];
    \draw[fill=white] (0, 1) circle [radius=2pt];
    \draw[fill=white] (-3, 2) circle [radius=2pt];
    \draw[fill=white] (0, 2) circle [radius=2pt];
    \draw[fill=white] (3, 2) circle [radius=2pt];
    \draw[fill=white] (-6, 3) circle [radius=2pt];
    \draw[fill=white] (-3, 3) circle [radius=2pt];
    \draw[fill=white] (0, 3) circle [radius=2pt];
    \draw[fill=white] (3, 3) circle [radius=2pt];
    \draw[fill=white] (-6, 4) circle [radius=2pt];
    \draw[fill=white] (-3, 4) circle [radius=2pt];
    \draw[fill=white] (0, 4) circle [radius=2pt];
    \draw[fill=white] (-3, 5) circle [radius=2pt];
    \draw[fill=white] (-3, 6) circle [radius=2pt];

    \draw (-0.6, 0) node {\scalebox{0.5}{$\begin{pmatrix} 0 & 0 & 0 \\ 0 & 0 & 0 \end{pmatrix} $}};
    \draw (1, 0.8) node {\scalebox{0.5}{$\begin{pmatrix}  0.5 & 0.5 & 0.5 \\ 0 & 0 & 0 \end{pmatrix} $}};
    \draw (3.8, 2) node {\scalebox{0.5}{$\begin{pmatrix} 0.5 & 0.5 & 0.5 \\ 0.5 & 0.5 & 0 \end{pmatrix} $}};
    \draw (3.8, 3) node {\scalebox{0.5}{$\begin{pmatrix} 1 & 1 & 0.5 \\ 0.5 & 0.5 & 0 \end{pmatrix} $}};
    \draw (0.8, 4.3) node {\scalebox{0.5}{$\begin{pmatrix} 1 & 1 & 0.5 \\ 0.5 & 0.5 & 0.5 \end{pmatrix} $}};
    \draw (-3.8, 5.3) node {\scalebox{0.5}{$\begin{pmatrix} 1 & 1 & 1 \\ 0.5 & 0.5 & 0.5 \end{pmatrix} $}};
    \draw (-2.4, 6) node {\scalebox{0.5}{$\begin{pmatrix} 1 & 1 & 1 \\ 1 & 1 & 1 \end{pmatrix} $}};

    \draw (-3.8, 1.7) node {\scalebox{0.5}{$\begin{pmatrix} 0.5 & 0.5 & 0.5 \\ 0 & 0 & 0.5 \end{pmatrix} $}};
    \draw (-6.85, 3) node {\scalebox{0.5}{$\begin{pmatrix} 0.5 & 0.5 & 1 \\ 0 & 0 & 0.5 \end{pmatrix} $}};
    \draw (-6.7, 4) node {\scalebox{0.5}{$\begin{pmatrix} 1 & 1 & 1 \\ 0 & 0 & 0.5 \end{pmatrix} $}};

    \draw (-0.7, 1.7) node {\scalebox{0.5}{$\begin{pmatrix} 1 & 1 & 0.5 \\ 0 & 0 & 0 \end{pmatrix} $}};
    \draw (0.9, 3.2) node {\scalebox{0.5}{$\begin{pmatrix} 0.5 & 0.5 & 0.5 \\ 0.5 & 0.5 & 0.5 \end{pmatrix}$ }};

    \draw (-3.7, 2.7) node {\scalebox{0.5}{$\begin{pmatrix} 1 & 1 & 0.5 \\ 0 & 0 & 0.5 \end{pmatrix} $}};
    \draw (-2.1, 4.2) node {\scalebox{0.5}{$\begin{pmatrix} 0.5 & 0.5 & 1 \\ 0.5 & 0.5 & 0.5 \end{pmatrix}$ }};
    
    \end{tikzpicture}
    \caption{The lattice of fuzzy rough sets for Example \ref{ex:distlattice}}
    \label{fig:lattice}
\end{figure}
\end{example}

\section*{Conclusions}
\label{sec:conc}
The properties of a poset formed by fuzzy rough sets depend strongly
both on the framework in which the approximations
are defined (t-norm $\odot$ - implicator $\vartriangleright$), and on the properties of the approximation space $(U,\theta)$.

The majority of our arguments work only under some
finiteness conditions imposed on the domain or range of the fuzzy reference
sets and of the relation $\theta$. We hope that these conditions can be
replaced with weaker ones (see e.g. \cite{St}) or with
conditions related to some topology defined on $U$.

In case of a finite universe or range set
$L$, we were able to show that $(\mathcal{FR}(U,L),\leq)$
is a lattice only for a similarity relation $\theta$ in a
particular context (min t-norm and S-implicator), by using property (D). It would be interesting to check if the
proof can be extended for fuzzy quasiorders or other types of relations.
Theorem \ref{thm:frs} seems to suggest that such a result can be obtained
even in a general context (of a t-norm and a related implicator) for
a t-similarity relation $\theta$ with some (strong) particular properties, even in the absence of the property (D). This can serve as a further research goal.

Even in conditions of Theorem \ref{thm:lattice}, the lattices formed by fuzzy
rough sets are not distributive in general - this is shown in Example \ref{ex:notdistr}
below. Hence an interesting question could be if these lattices have any
characteristic common properties. We can see that for some particular approximation spaces as in Example \ref{ex:notfrs}, we even obtain a particular
distributive lattice (a so-called double Stone lattice). Therefore, it makes
sense to ask under what conditions imposed on $(U,\theta)$ will we obtain a
distributive lattice $\mathcal{FR}(U,L)$.

\begin{example}
\label{ex:notdistr}
Let $U$, $L$, the similarity relation $\theta$
be as in Example \ref{ex:distlattice}, and let us consider the fuzzy rough sets $\alpha_{1},\alpha_{2}$ from
\ref{ex:distlattice} and $c=%
\begin{pmatrix}
0.5 & 0.5 & 0.5\\
0.5 & 0.5 & 0.5
\end{pmatrix}
$. We prove that $(\alpha_{1}\wedge\alpha_{2})\vee c\neq(\alpha_{1}\vee
c)\wedge(\alpha_{2}\vee c)$: 

\smallskip

Indeed, by Example \ref{ex:distlattice}, $\alpha_{1}\wedge\alpha_{2}=%
\begin{pmatrix}
0.25 & 0.25 & 0.5\\
0.1 & 0.1 & 0.5
\end{pmatrix}
<c$, and hence $(\alpha_{1}\wedge\alpha_{2})\vee c=c$. In view of Remark \ref{rem:infsup} we
have 
\smallskip
$\alpha_{1}\vee c=%
\begin{pmatrix}
1 & 0.75 & 0.5\\
0.5 & 0.5 & 0.5
\end{pmatrix}
$ and $\alpha_{2}\vee c=%
\begin{pmatrix}
0.75 & 1 & 0.5\\
0.5 & 0.5 & 0.5
\end{pmatrix}
$, because
$\alpha_{1}\vee c=%
\begin{pmatrix}
\overline{\theta}(h_{1})\\
\underline{\theta}(h_{1})
\end{pmatrix}
$ and $\alpha_{2}\vee c=%
\begin{pmatrix}
\overline{\theta}(h_{2})\\
\underline{\theta}(h_{2})
\end{pmatrix}
$, where $h_{1} = 1/a + 0.5/b + 0.5/c $ and  $h_{2} = 0.5/a + 1/b + 0.5/c $.

\medskip
Now, observe that $%
\begin{pmatrix}
\overline{\theta}(h_{1})\wedge\overline{\theta}(h_{2})\\
\underline{\theta}(h_{1})\wedge\underline{\theta}(h_{2})
\end{pmatrix}
=%
\begin{pmatrix}
0.75 & 0.75 & 0.5\\
0.5 & 0.5 & 0.5
\end{pmatrix}
$ is a fuzzy rough set induced by the fuzzy set $m = 0.75/a + 0.5/b + 0.5/c $.
In view of Remark \ref{rem:infsup} this means that $(\alpha_{1}\vee c)\wedge(\alpha
_{2}\vee c)=%
\begin{pmatrix}
0.75 & 0.75 & 0.5\\
0.5 & 0.5 & 0.5
\end{pmatrix}
\neq c$.
\end{example}


\bibliography{rough}

\begin{thebibliography}{29}
\expandafter\ifx\csname natexlab\endcsname\relax\def\natexlab#1{#1}\fi
\providecommand{\url}[1]{\texttt{#1}}
\providecommand{\href}[2]{#2}
\providecommand{\path}[1]{#1}
\providecommand{\DOIprefix}{doi:}
\providecommand{\ArXivprefix}{arXiv:}
\providecommand{\URLprefix}{URL: }
\providecommand{\Pubmedprefix}{pmid:}
\providecommand{\doi}[1]{\href{http://dx.doi.org/#1}{\path{#1}}}
\providecommand{\Pubmed}[1]{\href{pmid:#1}{\path{#1}}}
\providecommand{\bibinfo}[2]{#2}
\ifx\xfnm\relax \def\xfnm[#1]{\unskip,\space#1}\fi
\bibitem[{Boixader et~al.(2000)Boixader, Jacas and Recasens}]{BJR}
\bibinfo{author}{Boixader, D.}, \bibinfo{author}{Jacas, J.},
  \bibinfo{author}{Recasens, J.}, \bibinfo{year}{2000}.
\newblock \bibinfo{title}{Upper and lower approximations of fuzzy sets}.
\newblock \bibinfo{journal}{International Journal of General Systems}
  \bibinfo{volume}{29}, \bibinfo{pages}{555--568}.
\newblock \DOIprefix\doi{10.1080/03081070008960961}.
\bibitem[{del Cerro and Prade(1986)}]{CP}
\bibinfo{author}{del Cerro, L.F.}, \bibinfo{author}{Prade, H.},
  \bibinfo{year}{1986}.
\newblock \bibinfo{title}{Rough sets, twofold fuzzy sets and modal
  logic—fuzziness in indiscernibility and partial information}.
\newblock \bibinfo{journal}{The Mathematics of Fuzzy Systems}
  \bibinfo{volume}{88}, \bibinfo{pages}{103--120}.
\bibitem[{De~Cock et~al.(2007)De~Cock, Cornelis and Kerre}]{CCK}
\bibinfo{author}{De~Cock, M.}, \bibinfo{author}{Cornelis, C.},
  \bibinfo{author}{Kerre, E.E.}, \bibinfo{year}{2007}.
\newblock \bibinfo{title}{Fuzzy rough sets: The forgotten step}.
\newblock \bibinfo{journal}{IEEE Transactions on Fuzzy Systems}
  \bibinfo{volume}{15}, \bibinfo{pages}{121--130}.
\newblock \DOIprefix\doi{10.1109/TFUZZ.2006.889762}.
\bibitem[{D'eer et~al.(2015)D'eer, Verbiest, Cornelis and Godo}]{DVCG}
\bibinfo{author}{D'eer, L.}, \bibinfo{author}{Verbiest, N.},
  \bibinfo{author}{Cornelis, C.}, \bibinfo{author}{Godo, L.},
  \bibinfo{year}{2015}.
\newblock \bibinfo{title}{A comprehensive study of
  implicator–conjunctor-based and noise-tolerant fuzzy rough sets:
  Definitions, properties and robustness analysis}.
\newblock \bibinfo{journal}{Fuzzy Sets and Systems} \bibinfo{volume}{275},
  \bibinfo{pages}{1--38}.
\newblock \DOIprefix\doi{https://doi.org/10.1016/j.fss.2014.11.018}.
\bibitem[{Dubois and Prade(1990)}]{DP}
\bibinfo{author}{Dubois, D.}, \bibinfo{author}{Prade, H.},
  \bibinfo{year}{1990}.
\newblock \bibinfo{title}{Rough fuzzy sets and fuzzy rough sets}.
\newblock \bibinfo{journal}{International Journal of General Systems}
  \bibinfo{volume}{17}, \bibinfo{pages}{191--209}.
\newblock \DOIprefix\doi{10.1080/03081079008935107}.
\bibitem[{Fodor(2004)}]{Fodor}
\bibinfo{author}{Fodor, J.}, \bibinfo{year}{2004}.
\newblock \bibinfo{title}{Left-continuous t-norms in fuzzy logic: an overview}.
\newblock \bibinfo{journal}{Acta Polytechnica Hungarica} \bibinfo{volume}{1},
  \bibinfo{pages}{35--47}.
\bibitem[{Gottwald and Jenei(2001)}]{GJ}
\bibinfo{author}{Gottwald, S.}, \bibinfo{author}{Jenei, S.},
  \bibinfo{year}{2001}.
\newblock \bibinfo{title}{A new axiomatization for involutive monoidal
  t-norm-based logic}.
\newblock \bibinfo{journal}{Fuzzy Sets and Systems} \bibinfo{volume}{124},
  \bibinfo{pages}{303--307}.
\newblock \DOIprefix\doi{https://doi.org/10.1016/S0165-0114(01)00100-2}.
  \bibinfo{note}{fuzzy Logic}.
\bibitem[{Greco et~al.(1998)Greco, Matarazzo and S{\l}owinski}]{GMS}
\bibinfo{author}{Greco, S.}, \bibinfo{author}{Matarazzo, B.},
  \bibinfo{author}{S{\l}owinski, R.}, \bibinfo{year}{1998}.
\newblock \bibinfo{title}{Fuzzy similarity relation as a basis for rough
  approximations}, in: \bibinfo{booktitle}{Rough Sets and Current Trends in
  Computing}, \bibinfo{publisher}{Springer Berlin Heidelberg},
  \bibinfo{address}{Berlin, Heidelberg}. pp. \bibinfo{pages}{283--289}.
\bibitem[{Gégény et~al.(2020)Gégény, Kovács and Radeleczki}]{GKR}
\bibinfo{author}{Gégény, D.}, \bibinfo{author}{Kovács, L.},
  \bibinfo{author}{Radeleczki, S.}, \bibinfo{year}{2020}.
\newblock \bibinfo{title}{Notes on the lattice of fuzzy rough sets with crisp
  reference sets}.
\newblock \bibinfo{journal}{International Journal of Approximate Reasoning}
  \bibinfo{volume}{126}, \bibinfo{pages}{124--132}.
\newblock \DOIprefix\doi{https://doi.org/10.1016/j.ijar.2020.08.007}.
\bibitem[{Hu et~al.(2011)Hu, Yu, Pedrycz and Chen}]{HYPC}
\bibinfo{author}{Hu, Q.}, \bibinfo{author}{Yu, D.}, \bibinfo{author}{Pedrycz,
  W.}, \bibinfo{author}{Chen, D.}, \bibinfo{year}{2011}.
\newblock \bibinfo{title}{Kernelized fuzzy rough sets and their applications}.
\newblock \bibinfo{journal}{IEEE Transactions on Knowledge and Data
  Engineering} \bibinfo{volume}{23}, \bibinfo{pages}{1649--1667}.
\newblock \DOIprefix\doi{10.1109/TKDE.2010.260}.
\bibitem[{Inuiguchi(2004)}]{In}
\bibinfo{author}{Inuiguchi, M.}, \bibinfo{year}{2004}.
\newblock \bibinfo{title}{Classification-versus approximation-oriented fuzzy
  rough sets}, in: \bibinfo{booktitle}{Proceedings of Information Processing
  and Management of Uncertainty in Knowledge-Based Systems, IPMU}.
\bibitem[{J{\"a}rvinen et~al.(2013)J{\"a}rvinen, Pagliani and Radeleczki}]{JPR}
\bibinfo{author}{J{\"a}rvinen, J.}, \bibinfo{author}{Pagliani, P.},
  \bibinfo{author}{Radeleczki, S.}, \bibinfo{year}{2013}.
\newblock \bibinfo{title}{Information completeness in $\text{N}$elson algebras
  of rough sets induced by quasiorders}.
\newblock \bibinfo{journal}{Studia Logica} \bibinfo{volume}{101},
  \bibinfo{pages}{1073--1092}.
\bibitem[{J{\"a}rvinen and Radeleczki(2014)}]{JR1}
\bibinfo{author}{J{\"a}rvinen, J.}, \bibinfo{author}{Radeleczki, S.},
  \bibinfo{year}{2014}.
\newblock \bibinfo{title}{Rough sets determined by tolerances}.
\newblock \bibinfo{journal}{International Journal of Approximate Reasoning}
  \bibinfo{volume}{55}, \bibinfo{pages}{1419--1438}.
\newblock \DOIprefix\doi{10.1016/j.ijar.2013.12.005}.
\bibitem[{J{\"a}rvinen et~al.(2009)J{\"a}rvinen, Radeleczki and Veres}]{JRV09}
\bibinfo{author}{J{\"a}rvinen, J.}, \bibinfo{author}{Radeleczki, S.},
  \bibinfo{author}{Veres, L.}, \bibinfo{year}{2009}.
\newblock \bibinfo{title}{Rough sets determined by quasiorders}.
\newblock \bibinfo{journal}{Order} \bibinfo{volume}{26},
  \bibinfo{pages}{337--355}.
\newblock \DOIprefix\doi{10.1007/s11083-009-9130-z}.
\bibitem[{Liu(2008)}]{Liu}
\bibinfo{author}{Liu, G.}, \bibinfo{year}{2008}.
\newblock \bibinfo{title}{Axiomatic systems for rough sets and fuzzy rough
  sets}.
\newblock \bibinfo{journal}{International Journal of Approximate Reasoning}
  \bibinfo{volume}{48}, \bibinfo{pages}{857--867}.
\newblock \DOIprefix\doi{https://doi.org/10.1016/j.ijar.2008.02.001}.
\bibitem[{Ma and Hu(2013)}]{MH}
\bibinfo{author}{Ma, Z.M.}, \bibinfo{author}{Hu, B.Q.}, \bibinfo{year}{2013}.
\newblock \bibinfo{title}{Topological and lattice structures of l-fuzzy rough
  sets determined by lower and upper sets}.
\newblock \bibinfo{journal}{Information Sciences} \bibinfo{volume}{218},
  \bibinfo{pages}{194--204}.
\newblock \DOIprefix\doi{https://doi.org/10.1016/j.ins.2012.06.029}.
\bibitem[{Mi and Zhang(2004)}]{MZ}
\bibinfo{author}{Mi, J.S.}, \bibinfo{author}{Zhang, W.X.},
  \bibinfo{year}{2004}.
\newblock \bibinfo{title}{An axiomatic characterization of a fuzzy
  generalization of rough sets}.
\newblock \bibinfo{journal}{Information Sciences} \bibinfo{volume}{160},
  \bibinfo{pages}{235--249}.
\newblock \DOIprefix\doi{https://doi.org/10.1016/j.ins.2003.08.017}.
\bibitem[{Morsi and Yakout(1998)}]{MY}
\bibinfo{author}{Morsi, N.}, \bibinfo{author}{Yakout, M.},
  \bibinfo{year}{1998}.
\newblock \bibinfo{title}{Axiomatics for fuzzy rough sets}.
\newblock \bibinfo{journal}{Fuzzy Sets and Systems} \bibinfo{volume}{100},
  \bibinfo{pages}{327--342}.
\newblock \DOIprefix\doi{https://doi.org/10.1016/S0165-0114(97)00104-8}.
\bibitem[{Nakamura(1988)}]{NK}
\bibinfo{author}{Nakamura, A.}, \bibinfo{year}{1988}.
\newblock \bibinfo{title}{Fuzzy rough sets}.
\newblock \bibinfo{journal}{Note on Multiple-Valued Logic in Japan}
  \bibinfo{volume}{9}, \bibinfo{pages}{1--8}.
\bibitem[{Pawlak(1982)}]{Pawlak1}
\bibinfo{author}{Pawlak, Z.}, \bibinfo{year}{1982}.
\newblock \bibinfo{title}{Rough sets}.
\newblock \bibinfo{journal}{International Journal of Computer \& Information
  Sciences} \bibinfo{volume}{11}, \bibinfo{pages}{341--356}.
\newblock \DOIprefix\doi{10.1007/BF01001956}.
\bibitem[{Pei(2005)}]{Pei}
\bibinfo{author}{Pei, D.}, \bibinfo{year}{2005}.
\newblock \bibinfo{title}{A generalized model of fuzzy rough sets}.
\newblock \bibinfo{journal}{International Journal of General Systems}
  \bibinfo{volume}{34}, \bibinfo{pages}{603--613}.
\newblock \DOIprefix\doi{10.1080/03081070500096010}.
\bibitem[{Pomyka{\l}a and Pomyka{\l}a(1988)}]{Po88}
\bibinfo{author}{Pomyka{\l}a, J.}, \bibinfo{author}{Pomyka{\l}a, J.A.},
  \bibinfo{year}{1988}.
\newblock \bibinfo{title}{The {S}tone algebra of rough sets}.
\newblock \bibinfo{journal}{Bulletin of Polish Academy of Sciences.
  Mathematics} \bibinfo{volume}{36}, \bibinfo{pages}{495--–512}.
\bibitem[{Radzikowska and Kerre(2002)}]{RK}
\bibinfo{author}{Radzikowska, A.M.}, \bibinfo{author}{Kerre, E.E.},
  \bibinfo{year}{2002}.
\newblock \bibinfo{title}{A comparative study of fuzzy rough sets}.
\newblock \bibinfo{journal}{Fuzzy Sets and Systems} \bibinfo{volume}{126},
  \bibinfo{pages}{137--155}.
\newblock \DOIprefix\doi{https://doi.org/10.1016/S0165-0114(01)00032-X}.
\bibitem[{Radzikowska and Kerre(2005)}]{RKL}
\bibinfo{author}{Radzikowska, A.M.}, \bibinfo{author}{Kerre, E.E.},
  \bibinfo{year}{2005}.
\newblock \bibinfo{title}{Fuzzy rough sets based on residuated lattices}, in:
  \bibinfo{booktitle}{Transactions on Rough Sets II},
  \bibinfo{publisher}{Springer Berlin}. pp. \bibinfo{pages}{278--296}.
\bibitem[{Stout(2001)}]{St}
\bibinfo{author}{Stout, L.N.}, \bibinfo{year}{2001}.
\newblock \bibinfo{title}{Finiteness notions in fuzzy sets}.
\newblock \bibinfo{journal}{Fuzzy Sets and Systems} \bibinfo{volume}{124},
  \bibinfo{pages}{25--33}.
\newblock \DOIprefix\doi{https://doi.org/10.1016/S0165-0114(01)00004-5}.
\bibitem[{Wu et~al.(2005)Wu, Leung and Mi}]{WLM}
\bibinfo{author}{Wu, W.Z.}, \bibinfo{author}{Leung, Y.}, \bibinfo{author}{Mi,
  J.S.}, \bibinfo{year}{2005}.
\newblock \bibinfo{title}{On characterizations of (i,t)-fuzzy rough
  approximation operators}.
\newblock \bibinfo{journal}{Fuzzy Sets and Systems} \bibinfo{volume}{154},
  \bibinfo{pages}{76--102}.
\newblock \DOIprefix\doi{https://doi.org/10.1016/j.fss.2005.02.011}.
\bibitem[{Wu et~al.(2003)Wu, Mi and Zhang}]{WMZ}
\bibinfo{author}{Wu, W.Z.}, \bibinfo{author}{Mi, J.S.}, \bibinfo{author}{Zhang,
  W.X.}, \bibinfo{year}{2003}.
\newblock \bibinfo{title}{Generalized fuzzy rough sets}.
\newblock \bibinfo{journal}{Information Sciences} \bibinfo{volume}{151},
  \bibinfo{pages}{263--282}.
\newblock \DOIprefix\doi{https://doi.org/10.1016/S0020-0255(02)00379-1}.
\bibitem[{Yao(1998)}]{YY}
\bibinfo{author}{Yao, Y.}, \bibinfo{year}{1998}.
\newblock \bibinfo{title}{Generalized rough set models}.
\newblock \bibinfo{journal}{Rough sets in knowledge discovery}
  \bibinfo{volume}{1}, \bibinfo{pages}{286--318}.
\bibitem[{Zadeh(1965)}]{Zh}
\bibinfo{author}{Zadeh, L.A.}, \bibinfo{year}{1965}.
\newblock \bibinfo{title}{Fuzzy sets}.
\newblock \bibinfo{journal}{Information and Control} \bibinfo{volume}{8},
  \bibinfo{pages}{338--353}.
\newblock \DOIprefix\doi{10.1016/S0019-9958(65)90241-X}.

\end{thebibliography}
\bibliographystyle{elsarticle-harv}

\end{document}